\documentclass[11pt]{amsart}
\usepackage{amsmath}
\usepackage{amscd}
\usepackage{amssymb}
\usepackage{amsfonts}
\usepackage{amsthm}
\usepackage{bbm}
\usepackage{cancel}
\usepackage{color}
\usepackage{eucal}
\usepackage{enumerate,yfonts}
 
\usepackage{fullpage}
\usepackage{pdfsync}
\usepackage[all,cmtip]{xy}
\usepackage{graphicx}
\usepackage{graphics}
\usepackage{hyperref}
\usepackage{latexsym}
\usepackage{mathrsfs}
\usepackage{placeins}
\usepackage{pstricks}
\usepackage{bookmark}

\usepackage{stmaryrd}
\usepackage{url}

\newtheorem{thm}{Theorem}[section]
\newtheorem{corollary}[thm]{Corollary}
\newtheorem{lemma}[thm]{Lemma}

\newtheorem{proposition}[thm]{Proposition}
\newtheorem{prop}[thm]{Proposition}
\newtheorem{conjecture}[thm]{Conjecture}
\newtheorem{thm-dfn}[thm]{Theorem-Definition}

\theoremstyle{definition}
\newtheorem{definition}[thm]{Definition}

\numberwithin{equation}{section}

\theoremstyle{remark}
\newtheorem{remark}{Remark}[section]
\newtheorem{example}[remark]{Example}

\newcommand{\fg}{{\mathfrak g}}
\newcommand{\ft}{{\mathfrak t}}

\newcommand{\fb}{{\mathfrak b}}

\newcommand{\fn}{{\mathfrak n}}

\newcommand{\frakM}{{\mathfrak M}}

\newcommand{\rW}{{\mathrm W}}
\newcommand{\rU}{{\mathrm U}}

\newcommand{\bC}{{\mathbb C}}
\newcommand{\bX}{{\mathbb X}}
\newcommand{\bG}{{\mathbb G}}
\newcommand{\bZ}{{\mathbb Z}}

\newcommand{\bQ}{{\mathbb Q}}

\newcommand{\mD}{\mathcal{D}}
\newcommand{\mS}{\mathcal{S}}

\newcommand{\mE}{\mathcal{E}}
\newcommand{\mF}{\mathcal{F}}

\newcommand{\mA}{\mathcal{A}}
\newcommand{\mM}{\mathcal{M}}
\newcommand{\mT}{\mathcal{T}}
\newcommand{\mO}{\mathcal{O}}
\newcommand{\mL}{\mathcal{L}}
\newcommand{\mH}{\mathcal{H}}

\newcommand{\mG}{\mathcal{G}}
\newcommand{\mN}{\mathcal{N}}
\newcommand{\mB}{\mathcal{B}}

\newcommand{\calC}{{\mathcal C}}

\newcommand{\calF}{{\mathcal F}}
\newcommand{\calG}{{\mathcal G}}

\newcommand{\calL}{{\mathcal L}}
\newcommand{\calM}{{\mathcal M}}
\newcommand{\calN}{{\mathcal N}}

\newcommand{\calR}{{\mathcal R}}
\newcommand{\calS}{{\mathcal S}}
\newcommand{\calT}{{\mathcal T}}

\newcommand{\calX}{{\mathcal X}}
\newcommand{\calY}{{\mathcal Y}}
\newcommand{\calZ}{{\mathcal Z}}

\newcommand{\cO}{{\mathcal O}}
\newcommand{\cA}{{\mathcal A}}
\newcommand{\cF}{{\mathcal F}}
\newcommand{\cN}{{\mathcal N}}

\newcommand{\cE}{{\mathcal E}}

\newcommand{\cZ}{{\mathcal Z}}
\newcommand{\cS}{{\mathcal{S}}}
\newcommand{\cM}{{\mathcal{M}}}
\newcommand{\cG}{{\mathcal{G}}}

\newcommand{\sM}{\mathscr{M}}

\newcommand{\sH}{\mathscr{H}}

\newcommand{\sD}{\mathscr{D}}

\newcommand{\on}{\operatorname}

\newcommand{\tU}{\widetilde{\on{U}}}

\newcommand{\ra}{\rightarrow}
\newcommand{\la}{\leftarrow}

\newcommand{\is}{\simeq}

\newcommand{\oH}{{\mathrm H}}

\newcommand{\Loc}{\on{LocSys}}

\newcommand{\nc}{\newcommand}

\nc{\al}{{\alpha}} \nc{\be}{{\beta}} \nc{\ga}{{\gamma}}
\nc{\ve}{{\varepsilon}} \nc{\Ga}{{\Gamma}} 
\nc{\La}{{\Lambda}}

\nc{\ad }{{\on{ad }}}

\nc{\aff}{{\on{aff}}} \nc{\Aff}{{\mathbf{Aff}}}

\nc{\der}{{\on{der}}}

\nc{\diag}{{\on{diag}}}
\newcommand{\End}{{\on{End}}}
\nc{\Fl}{{\calF\ell}}

\nc{\Hg}{{\on{Higgs}}}
\newcommand{\Hom}{{\on{Hom}}}
\newcommand{\id}{{\on{id}}}
\nc{\Id}{{\on{Id}}}

\nc{\Ind}{{\on{Ind}}}

\nc{\Op}{{\on{Op}}}

\newcommand{\pr}{{\on{pr}}}
\newcommand{\Res}{{\on{Res}}}
\nc{\res}{{\on{res}}}

\newcommand{\Spec}{{\on{Spec}}}

\nc{\tr}{{\on{tr}}}
\newcommand{\Tr}{{\on{Tr}}}

\newcommand{\GL}{{\on{GL}}}
\nc{\GSp}{{\on{GSp}}} \nc{\GU}{{\on{GU}}} \nc{\SL}{{\on{SL}}}
\nc{\SU}{{\on{SU}}} \nc{\SO}{{\on{SO}}}
\newcommand{\Ad}{{\on{Ad}}}

\nc{\nh}{{\Loc_{J^p}(\tau')}}
\nc{\bnh}{{\Loc_{\breve J^p}(\tau')}}

\nc{\bU}{{\overline{U}}} 
\nc{\IC}{{\on{IC}}}

\newcommand{\br}{\begin{rouge}}
\newcommand{\er}{\end{rouge}}

\newcommand{\bb}{\begin{bluet}}
\newcommand{\eb}{\end{bluet}}

\newcommand{\U}{\on{U}}

\newcommand{\lra}{\longrightarrow}

\newcommand{\ul}{{\underline\lambda}}

\newcommand{\rs}{\on{rs}}

\newcommand{\barQ}{\overline{\mathbb Q}_\ell}

\newcommand{\Av}{\on{Av}}

\newcommand{\prolim}{\textup{}\varprojlim\textup{}}

\nc{\ot}{\otimes}

\nc{\oh}{{\operatorname{H}}}
\nc{\gr}{{\operatorname{gr}}}
\nc{\rk}{{\operatorname{rank}}}
\nc{\codim}{{\operatorname{codim}}}
\nc{\img}{{\operatorname{Im}}}
\nc{\Span}{{\operatorname{Span}}}
\nc{\Img}{\operatorname{Im}}

\newcommand{\beqn}{\begin{equation*}}
\newcommand{\eeqn}{\end{equation*}}

\newcommand{\beq}{\begin{equation}}
\newcommand{\eeq}{\end{equation}}

\newcommand{\bern}{\begin{eqnarray*}}
\newcommand{\eern}{\end{eqnarray*}}
\nc{\Fano}[1]{\on{Fano}_{#1}}

\newcommand{\lL}{\lambda+\Lambda}
\newcommand{\hlL}{\widehat{\lambda+\Lambda}}
\newcommand{\HC}{\mathcal{HC}}

\newcommand{\sign}{\on{sign}}

\newcommand{\quash}[1]{}  

\quash{
\setlength{\parskip}{1ex}
\setlength{\oddsidemargin}{0in}
\setlength{\evensidemargin}{0in}
\setlength{\textwidth}{6.5in}
\setlength{\topmargin}{-0.15in}
\setlength{\textheight}{8.6in}

}

\newenvironment{rouge}
{\relax\color{red}}
{\hspace*{.3ex}\relax}

\newenvironment{bluet}
{\relax\color{blue}}
{\hspace*{.3ex}\relax}

\begin{document}
\title{On the conjectures of  
Braverman-Kazhdan}
        \author{Tsao-Hsien Chen}
        \address{School of Mathematics,
      University of Minnesota, Twin Cities
          }
        \email{chenth@umn.edu}

\begin{abstract}
In this article we prove a conjecture of Braverman and Kazhdan 
in \cite{BK1} on acyclicity of $\rho$-Bessel sheaves on reductive groups
in both $\ell$-adic and de Rham settings. We do so by 
establishing a vanishing conjecture 
proposed in \cite{C1}.
As a corollary, we obtain a geometric construction of the 
non-linear Fourier kernels for finite reductive groups as conjectured by 
Braverman and Kazhdan.
The proof of the vanishing conjecture 
relies on the techniques developed in \cite{BFO} on Drinfeld center of Harish-Chandra bimodules and character D-modules, and 
a construction of a class of character sheaves in 
mixed-characteristic.

\quash{

In the previous work \cite{C}, we proposed a vanishing conjecture 
for a certain class of $\ell$-adic sheaves on a reductive group,
which contains  
the Braverman-Kazhdan conjecture in \cite{BK} on acyclicity
of 
cohomology vanishing of gamma sheaves, and we established the 
vanishing conjecture for $\GL_n$ using mirabolic subgroups.
In this article, we prove the vanishing conjecture for general reductive groups and also its
de Rham counterpart. In particular, we prove the 
Braverman-Kazhdan conjecture.

The proof 
relies on the techniques developed in \cite{BFO} on
Drinfeld center of Harish-Chandra bimodules and character $D$-modules. 
As an application, we
show that the functors of convolution with gamma $D$-modules, which can be 
viewed as a version of 
\emph{non-linear Fourier transforms}, 
commute with induction functors and 
are exact on the category of 
admissible $D$-modules on a reductive group.}
\end{abstract}

\maketitle

\setcounter{tocdepth}{1} 
\tableofcontents

\section{Introduction}
This paper is a sequel to \cite{C1}. In \emph{loc.cit.} it was shown that 
the (generalized) Braverman-Kazhdan conjecture on
acyclicity of $\rho$-Bessel sheaves on reductive groups
follows from a certain 
vanishing conjecture.
The goal of this paper is to give a proof of this vanishing conjecture.

In the introduction, we would like to recall the statement of the vanishing conjecture 
in both $\ell$-adic and de Rham settings, explain its applications to the
Braverman-Kazhdan conjectures, and outline a proof of the vanishing conjecture.

\subsection{The vanishing conjecture}
Let $k$ be an algebraically closure of a finite field 
$\mathbb F_q$ with $q$-element of characteristic $p>0$
or $k=\bC$. 
We fix a prime number $\ell$ different from $p$. 
We set $F=\barQ$ in the case $\on{char} k=p$ and $F=\bC$ in the case $k=\bC$. 
We will consider the following two geometric/sheaf-theoretic contexts:
(1)  $\ell$-adic sheaves on schemes over $k$ of characteristic $p$ (2) holonomic $D$-modules on
schemes over $k=\bC$.
We refer context (1) as the $\ell$-adic setting and context (2) as the de Rham setting.
We will fix a non-trivial character $\psi:\mathbb F_q\to\barQ^\times$. 
Depending on the setting, let $\mL_\psi$ be the 
Artin-Schreier sheaf on the additive group $\bG_a$
corresponding to $\psi$ in the $\ell$-adic setting or the exponential $D$-module 
in the de Rham setting.

The vanishing conjecture proposed in \cite{C1}
is a generalization of the well-known acyclicity 
 \beq\label{AS}
\oH_c^*(\bG_a,\mL_\psi)=0.
\eeq
of $\mL_\psi$ to general reductive groups.
The starting point is the observation that~\eqref{AS} can be restated as  
acyclicity of a certain local system on $\SL_2$ over certain affine subspaces.
Namely, let $\on{tr}:\SL_2\to\bG_a$ be the trace map and
let $U$ be the unipotent radical of the 
the standard Borel subgroup $B$ of $\SL_2$.
Then
~\eqref{AS} is equivalent to
the following acyclicity of the local system $\Phi=\tr^*\mL_\psi$ over $U$-orbits on the open Bruhat cell: for any 
$x\in \SL_2\setminus B$ we have 
\beq\label{reformulation}
\oH_c^*(xU,i^*\Phi)=0.
\eeq
Here $i:xU\to\SL_2$ is the embedding. 
Indeed, it follows from the fact that 
for any $x\in \SL_2\setminus B$, the trace map restricts to an isomorphism 
$\on{tr}:xU\is\bG_a$ between the $U$-orbit through $x$ and $\bG_a$.

To state a
generalization of \eqref{reformulation} to
general reductive groups, let me first recall some 
notations and definitions.
Let $G$ be a connected reductive group over 
$k$.
Let $T$ be a maximal tour of $G$ and 
$B$ be a Borel subgroup containing $T$ with 
unipotent radical $U$. 
Denote by
$\rW=N_G(T)/T$ the Weyl group, where $N_G(T)$
is the normalizer of $T$ in $G$.
Depending on the setting, we denote by
$\pi_1^t(T)$ the tame \'etale fundamental group of $T$ if $\on{char}k>0$, or
the topological fundamental group of $T$ if $k=\bC$. 
We denote by
$\calC(T)(F)$ the set of continuous $F$-valued characters of $\pi_1^t(T)$.
 
For any character $\chi\in\calC(T)(F)$,
we write 
$\mL_\chi$ for the corresponding rank one $\ell$-adic/de Rham local system on
$T$.
The Weyl group $\rW$ acts naturally on $\calC(T)(F)$ and for any $\chi\in\calC(T)(F)$, 
we denote by $\rW_\chi'$ the stabilizer of 
$\chi$ in $\rW$ and 
$\rW_\chi\subset\rW'_\chi$, the subgroup of $\rW_\chi'$ generated by those reflections 
$s_\alpha$ such that the pull-back 
$(\check\alpha)^*\mL_\chi$ is isomorphic to the trivial local system, where
$\check\alpha:\bG_m\to T$ is the coroot 
associated 
to $\alpha$\footnote{
The group $\rW_\chi$ plays an important role in the study of 
representations of finite reductive groups and character sheaves (see, e.g., \cite{Lu}).}.

Denote by
$\sD_\rW(T)$ 
the $\rW$-equivariant bounded derived category of sheaves on $T$.
For any $\mF\in\sD_\rW(T)$
and $\chi\in\calC(T)(F)$, the 
$\rW$-equivariant structure on $\mF$
together with
the natural $\rW_\chi'$-equivivariant structure on $\mL_\chi$
give rise to an
action of 
$\rW'_\chi$ on the
cohomology groups
$\oH_c^*(T,\mF\otimes\mL_\chi)$ (resp. $\oH^*(T,\mF\otimes\mL_\chi)$).
In particular, we get an action of the subgroup 
$\rW_\chi\subset\rW_\chi'$ on the cohomology groups above.
Denote by $\sign_\rW:\rW\to\{\pm1\}$ the sign character of 
$\rW$.

The key in formulating the generalization of~\eqref{reformulation} to general $G$ is the 
following definition of central complexes on $T$
introduced in \cite[Definition 1.1]{C1}:
\begin{definition}\label{def of central}
A $\rW$-equivariant complex $\mF\in\sD_\rW(T)$ 
is called central (resp. $*$-central) if for any $\chi\in\calC(T)(F)$, the group
$\rW_\chi$ acts on 
\[\oH_c^*(T,\mF\otimes\mL_\chi)\ \ \ \ \ 
(resp.\ \ \oH^*(T,\mF\otimes\mL_\chi))\]
via the sign character $\sign_\rW$.
It is called strongly central (resp. strongly $*$-central) if the stabilizer $\rW_\chi'$
acts on the cohomology groups above by the sign character.
\end{definition}

\begin{remark}\label{connected center}
If
the center of $G$ is connected, then it is known that 
$\rW_\chi=\rW_\chi'$ for all $\chi\in\calC(T)(F)$ (see, for example, \cite[Theorem 5.13]{DL}), thus
the notions of central complexes (resp. $*$-central complezes) and strongly central complexes (resp. strongly $*$-central complexes) are the same. In general, the two notions are different (see Example \ref{SL_2 example}). 
\end{remark}

\begin{example}\label{SL_2 example}
Consider the case $G=\SL_2$.
Let $\chi\in\calC(T)(F)$ be the 
 quadratic character associated to the covering $T\to T, x\to x^2$.
 We have $\rW_\chi=e$, $\rW'_\chi=\rW$, and 
 it is easy to see 
 that the $\rW$-equivariant local system $\mF=\mL_\chi$ corresponding to $\chi$ is central but not strongly central.

\end{example}

Consider the induction functor $\Ind_{T\subset B}^G:\sD(T)\to \sD(G)$
between bounded derived category of sheaves on $T$ and $G$.
For $\mF\in\sD_\rW(T)$, the $\rW$-equivariant 
structure on $\mF$ defines a $\rW$-action on
$\Ind_{T\subset B}^G(\mF)$ and we denote by
\[\Phi_\mF:=\Ind_{T\subset B}^G(\mF)^\rW\]
the $\rW$-invariant factor in $\Ind_{T\subset B}^G(\mF)$. 
In \cite[Conjecture 1.2]{C1}, we proposed the the following conjecture on
acyclicity of 
$\Phi_{\mF}$ over certain affine subspaces in $G$, called the vanishing conjecture:
\begin{conjecture}\label{vanishing conj}
Assume $\mF\in\sD_\rW(T)$ is central (resp. $*$-central).
For any $x\in G\setminus B$, 
we have the following cohomology vanishing
\beq\label{statement}
\oH_c^*(xU,i^*\Phi_{\mF})=0\ \ \ \ \ \ (\text{resp.}\ \ \oH^*(xU,i^!\Phi_{\mF})=0)
\eeq
where $i:xU\to G$ is the natural inclusion map.
Equivalently, the derived push-forward
$\pi_!(\Phi_\mF)$ (resp. $\pi_*(\Phi_\mF)$) is supported on the closed subset 
$T=B/U\subset G/U$. Here $\pi:G\to G/U$ is the quotient map.

\end{conjecture}

\begin{example}
Let $G=\SL_2$ and 
let $\tr_T:T\is\bG_m\to\bG_a, t\to t+t^{-1}$ be the trace map. Then it 
is easy to show 
that
the pull-back
$\mF=\tr_T^*\mL_\psi$ together with the canonical $\rW$-equivaraint structure
is central, moreover, we have  
\[\Phi_\mF\is\tr^*\mL_\psi\] where $\tr:\SL_2\to\bG_a$ is the trace map
(see, e.g., \cite[Example 1.2 and 1.7]{C1}). 
It follows that
Conjecture \ref{vanishing conj} (in this case) is equivalent to~\eqref{reformulation}, and hence is also equivalent to the 
acyclicity 
of Artin-Schreier sheaf~\eqref{AS}.

\end{example}

\subsection{Braverman-Kazhdan conjectures} 
Assume $k=\overline{\mathbb F}_q$ and 
$G$ is defined over $\mathbb F_q$.
In \cite{BK1,BK2}, Braverman and Kazhdan associated to each representation 
$\rho:\check G\to\GL(V_\rho)$ of the complex dual group, 
a $\barQ$-valued function 
\beq
\gamma_{G,\rho,\psi}:\on{Irr}(G(\mathbb F_q))\to\barQ
\eeq
on the set of irreducible representation of the finite group $G(\mathbb F_q)$, 
satisfying the following remarkable properties: 
\begin{enumerate}
\item
it is constant on Deligne-Lusztig packets, that is, we have
$\gamma_{G,\rho,\psi}(\pi)=\gamma_{G,\rho,\psi}(\pi')$
if $\pi$ and $\pi'$ appear in the same Deligne-Lusztig representation 
$\on{R}_{T,\theta}$, 
\item if $\pi$ appears in $\on{R}_{T,\theta}$, 
then the value $\gamma_{G,\rho,\psi}(\pi)$
is given by a certain explicit Gauss-type sum
associated to the character $\theta$.
\end{enumerate}
They called $\gamma_{G,\rho,\psi}$
the $\gamma$-function associated to $\rho$. 

The function $\gamma_{G,\rho,\psi}$ on $\on{Irr}(G(\mathbb F_q))$
gives rise to a $\barQ$-valued class function 
$\phi_{G,\rho,\psi}$ on $G(\mathbb F_q)$ 
characterized by the property that 
the operator 
\[\mathscr F_\rho:\on{Func}(G(\mathbb F_q))\to\on{Func}(G(\mathbb F_q))\]
on the space of functions on $G(\mathbb F_q)$
given by convolution with $\phi_{G,\rho,\psi}$ satisfies 
\[\mathscr F_\rho(\chi_\pi)=\gamma_{G,\rho,\psi}(\pi)\chi_\pi,\]
where $\chi_\pi$ is the character of $\pi\in\on{Irr}(G(\mathbb F_q))$.
In the case $G=\GL_n$ and $\rho=\on{std}$ is the standard representation of 
$\check G=\GL_n(\bC)$, the function 
$\phi_{G,\rho,\psi}$ is given by $\psi\circ\tr$ (up to some power of $q$) and
the operator $\mathscr F_{\rho}$ is 
the linear Fourier transform on the space of functions on $\GL_n(\mathbb F_q)$ (or rather, 
the restriction of the linear Fourier transform on $\on{Func}(\mathfrak{gl}_n(\mathbb F_q))$ to the subspace  
$\on{Func}(\GL_n(\mathbb F_q)$). Thus, one can view $\mathscr F_\rho$ as a kind of 
non-linear Fourier transform and $\phi_{G,\rho,\psi}$ as the corresponding Fourier kernel.

In \emph{loc. cit.}
Braverman and Kazhdan proposed a geometric construction of 
$\phi_{G,\rho,\psi}$ using the theory of $\ell$-adic sheaves.
To explain their construction, let us fix  
a $\on{F}$-stable maximal torus $T\subset G$
where $\on{F}:G\to G$ is the geometric Frobenius 
morphism and 
consider the restriction of $\rho$ to the dual maximal torus $\check T\subset\check G$.
Then there exists 
a collection of weights 
\[\underline\lambda=\{\lambda_1,...,\lambda_r\}\subset\bX^\bullet(\check T):=\Hom(\check T,\bC^\times)\]
such that there is an eigenspace decomposition 
$V_\rho=\bigoplus_{i=1}^r V_{\lambda_i}$
of $V_\rho$, where $\check T$ acts on $V_{\lambda_i}$ via the character $\lambda_i$.
One can regard $\underline\lambda$ as collection of 
co-characters of $T$ using the the canonical isomorphism
$\bX^\bullet(\check T)\is\bX_\bullet(T)$
and define
\[
\Phi_{T,\rho,\psi}=\pr_{\ul,!}\tr^*\mL_\psi[r] \ \ \ \ \ \ \ \ \ \ \Phi^*_{T,\rho,\psi}=\pr_{\ul,*}\tr^*\mL_\psi[r],
\]
where 
\[\pr_\ul:=\prod_{i=1}^r\lambda_i:\bG_m^r\longrightarrow T,\ \ \ \ \tr:\bG_m^r\longrightarrow\bG_a, (x_1,...,x_r)\ra\sum_{i=1}^r x_i.\]
It is shown in \cite{BK2} that both $\Phi_{T,\rho,\psi}$ and $\Phi^*_{T,\rho,\psi}$
carry natural 
$\rW$-equivariant structures and the resulting objects in 
$\sD_\rW(T)$, denote again by $\Phi_{T,\rho,\psi}$
and $\Phi^*_{T,\rho,\psi}$, are called the $\rho$-Bessel sheaves\footnote{In \cite{BK1,BK2}, the authors called
$\Phi_{T,\rho,\psi}$  $\gamma$-sheaves on $T$. However, based on the fact that 
the classical $\gamma$-function is the Mellin transform of the Bessel function, we follow \cite{Ng} and use the 
term $\rho$-Bessel sheaves instead of $\gamma$-sheaves}.
The $\rho$-Bessel sheaves on $G$, denoted by $\Phi_{G,\rho,\psi}$ and $\Phi^*_{G,\rho,\psi}$, are defined as 
\[\Phi_{G,\rho,\psi}=\Ind_{T\subset B}^G(\Phi_{T,\rho,\psi})^{\rW},\ \ \ \ \ \ \Phi^*_{G,\rho,\psi}=\Ind_{T\subset B}^G(\Phi^*_{T,\rho,\psi})^{\rW}.\]
It is shown in \cite[Theorem 4.2]{BK2} and \cite[Appendix B]{CN} that,
if $\rho$ satisfies certain positivity assumption (see \cite[Section 1.4]{BK2}),
then the $\rho$-Bessel sheaves $\Phi_{T,\rho,\psi}$ and $\Phi_{T,\rho,\psi}^*$ on $T$ are 
in fact local systems on the image of $\pr_\ul$, moreover, we have $\Phi_{T,\rho,\psi}\is\Phi^*_{T,\rho,\psi}$. 
This is a generalization of Deligne's theorem on Kloosterman sheaves \cite{De1}.

Braverman and Kazhdan showed that one can endow 
the $\rho$-Bessel sheaf $\Phi_{G,\phi,\psi}$ with a 
Weil structure 
$\on{F}^*\Phi_{G,\phi,\psi}\is\Phi_{G,\phi,\psi}$ and they proposed the following conjecture:
\begin{conjecture}\label{kernel}
Let $\Tr(\Phi_{G,\rho,\psi}):G(\mathbb F_q)\to\barQ$ is the function corresponding to
$\Phi_{G,\rho}$ via the functions-sheavers correspondence. We have
\[\Tr(\Phi_{G,\rho,\psi})=\phi_{G,\rho,\psi}\]
\end{conjecture}

Conjecture \ref{kernel} gives a geometric construction of the 
non-linear Fourier kernel $\phi_{G,\rho,\psi}$.
They also showed that Conjecture \ref{kernel} follows from the following conjecture on
acyclicity of $\rho$-Bessel sheaves:
\begin{conjecture}\cite[Conjecture 9.12]{BK1}\label{BK conj}
For any $x\in G\setminus B$, 
we have the following cohomology vanishing
\[\oH_c^*(xU,i^*\Phi_{G,\rho,\psi})=0\ \ \ \ (\text{resp.}\ \ 
\oH^*(xU,i^!\Phi^*_{G,\rho,\psi})=0)
\]
where $i:xU\to G$ is the natural inclusion map.
Equivalently, the derived push-forward
$\pi_!(\Phi_{G,\rho,\psi})$ (resp. $\pi_*(\Phi_{G,\rho,\psi})$) is supported on the closed subset 
$T=B/U\subset G/U$. Here $\pi:G\to G/U$ is the quotient map.
\end{conjecture}

The goal of this paper is to give a proof of 
Conjecture \ref{BK conj}, and hence Conjecture \ref{kernel}.

Note that the construction of $\rho$-Bessel sheaves 
and Conjecture \ref{BK conj} are entirely geometric and 
have obvious counterparts in the de Rham setting. 
Moreover, it is shown in \cite{C1} that 
the $\rho$-Bessel sheaves on $T$ are in fact strongly central
(see Definition \ref{def of central}). Thus
the vanishing conjecture contains Conjecture \ref{BK conj} as 
a special case 
and what we actually 
prove here is the 
vanishing conjecture (or rather, Conjecture \ref{vanishing conj}
for strongly central complexes). 


\begin{remark}
Conjecture \ref{kernel} and Conjecture \ref{BK conj} here are (slightly) generalized versions
of the original conjectures of Braverman and Kazhdan. The original conjectures require that the representation 
$\rho$ satisfies the positivity assumption mentioned earlier. 
In Corollary \ref{generalized BK conj} and Corollary \ref{BK conj functions} below, we will prove that 
their conjecture holds
without any assumption on $\rho$.

\end{remark}

\subsection{The main result}
The following theorem is the main result of the paper which 
establishes acyclicity of 
$\Phi_\mF$ 
when $\mF$ is a 
strongly central complex (resp. strongly $*$-central complex), and hence 
the vanishing conjecture for reductive groups with connected center (see Remark \ref{connected center}), 
for almost all characteristics:

\begin{thm}(Theorem \ref{main})\label{main result}
There exists a positive integer $N$ depending only on the type of the group $G$ such that 
the following holds. 
Assume $k=\bC$ or $\on{char}k=p$ is not dividing
$N\ell$.
Let $\mF\in\sD_\rW(T)$ be a strongly central complex (resp. strongly $*$-central complex) on $T$
and let $\Phi_\mF=\Ind_{T\subset B}^G(\mF)^\rW\in\sD(G)$.
Then for any $x\in G\setminus B$,
we have the following cohomology vanishing 
\beq\label{desired vanishing}
\oH_c^*(xU,i^*\Phi_\mF)=0\ \ \ \ \ \ \ \ (\text{resp.}\ \  \oH^*(xU,i^!\Phi_\mF)=0).
\eeq
Here $i:xU\to G$ is the embedding.
Equivalently, the derived push-forward
$\pi_!(\Phi_\mF)$ (resp. $\pi_*(\Phi_\mF)$) is supported on the closed subset 
$T=B/U\subset G/U$. Here $\pi:G\to G/U$ is the quotient map.

In particular,
the vanishing conjecture (Conjecture \ref{vanishing conj}) holds for reductive groups with
connected center. 
\end{thm}

\begin{remark}
The assumption on the characteristic of 
$k$ comes from a spreading out argument used in the proof (see Section \ref{sketch}).
\end{remark}
\begin{remark}
In \cite{C1}, we proved the vanishing conjecture in the case 
$G=\GL_n$ 
using mirabolic subgroups (note that $\GL_n$ has connected center). The argument in \emph{loc. cit.} was inspired by the 
work of Cheng and Ng\^o \cite{CN} on Braverman-Kazhdan conjectures for $G=\GL_n$.
The proof of Theorem \ref{main result} for general 
$G$ uses different methods (see Section \ref{sketch}).
\end{remark}

\subsection{Applications}\label{application}
In this subsection
we assume the characteristic of $k$ is either zero or not dividing $N\ell$, where 
$N$ is the positive integer in Theorem \ref{main result}.

\begin{corollary}\label{generalized BK conj}
Conjecture \ref{BK conj} holds.
\end{corollary}
\begin{proof}
It is shown in \cite[Theorem 1.4]{C1} that the Braverman-Kazhdan's $\rho$-Bessel sheaf  
$\Phi_{T,\rho,\psi}$ (resp. $\Phi_{T,\rho,\psi}^*$) on $T$ is strongly central (resp. strongly $*$-central). Thus Theorem \ref{main result} immediately implies the corollary.
\end{proof}

\begin{corollary}\label{BK conj functions}
Conjecture \ref{kernel} holds.
\end{corollary}
\begin{proof}
It is shown in \cite[Corollary 6.7]{BK2} that 
Conjecture \ref{BK conj} implies
Conjecture \ref{kernel}.
Thus Corollary \ref{generalized BK conj} implies Corollary \ref{BK conj functions}.
\end{proof}

Conjecture \ref{BK conj} was proved by 
Braverman and Kazhdan \cite[Theorem 6.9]{BK2}
in the case when the semi-simple rank of 
$G$ is less or equal to one, and by
Cheng and Ng\^o \cite[Theorem 2.4]{CN}
in the case $G=\GL_n$.

Conjecture \ref{kernel} was proved by Braverman and Kazhdan \cite[Theorem 1.6]{BK2}
when the semi-simple rank of 
$G$ is less or equal to one or $G=\GL_n$
under some assumption on $\rho$.
In a recent work, G. Laumon and E. Letellier established Conjecture \ref{kernel} via a different method \cite[Theorem 1.0.2]{LL}. It is interesting to note that in \emph{loc.cit.}
they also 
made no assumption on the representation $\rho$.

\begin{remark}
In \cite[Theorem 1.0.1]{LL}, Laumon-Letellier also proved a formula for the non-linear Fourier kernel $\phi_{G,\rho,\psi}$ in terms of Deligne-Lusztig inductions.
It will be interesting to prove a similar result in the de Rham setting, that is, 
write down an explicit formula
for the 
$\rho$-Bessel $D$-module  
$\Phi_{G,\rho,\psi}$, or rather, the corresponding 
system of differential equations on $G$. 
We expect applications of such formula
to the
Braverman-Kazhdan-Ng\^o's approach to 
functional equation of automorphic $L$-functions
\cite{BK1,Ng}.

\end{remark}

\subsection{Whittaker sheaves on $G$}
In the course of proving Theorem \ref{main result},
we establish several characterizations of $*$-central $D$-modules on $T$ using the Mellin transform, see Theorem \ref{characterization of central}.
On the other hand, in their works on 
Whittaker $D$-modules, nil Hecke algebras, and quantum Toda lattices, 
Ginzburg \cite{Gi2} and Lonergan \cite{L} 
prove that the 
category of Whittaker $D$-modules on $G$ is equivalent to
a certain full subcategory of the category of $\rW$-equivariant $D$-modules on $T$.
It turns out that, as an immediate corollary of Theorem \ref{characterization of central}, 
the latter full subcategory is equivalent to the category 
of central $D$-modules on $T$. 
Thus, combining the results in \emph{loc. cit.} we obtain:

\begin{thm}(Theorem \ref{Whit})\label{whit}
The category of Whittaker $D$-modules on $G$ is equivalent to
the abelian category of central $D$-modules on $T$. 
 
\end{thm}

It was mentioned in the introduction of \cite{Gi2} that
Drinfeld asked the question of finding a description of an 
$\ell$-adic counterpart of the category of Whittaker $D$-modules on $G$ in terms of 
$\rW$-equivariant sheaves on $T$. Theorem \ref{whit}
suggests that the category of 
$\ell$-adic
central perverse sheaves on
$T$ might provide an answer 
to Drinfeld's question (see Section \ref{Whittaker}).

\begin{remark}
In a recent work \cite{BZG}, Ben-Zvi and Gunningham 
constructed a remarkable functor from the 
category of Whittaker $D$-modules on $G$ to the category of 
$G$-conjugation equivaraint $D$-module on $G$, they called it 
Ng\^o functor, 
and they conjectured that the objects in the essential image of the Ng\^o functor satisfy the 
cohomology vanishing properties in~\eqref{statement}, see \cite[Conjecture 2.9 and 2.14]{BZG}.
Our Theorem \ref{main result} and Theorem \ref{whit} might be useful for
studying their conjectures.

\end{remark}


\subsection{Outline of the proof}\label{sketch}
The proof of Theorem \ref{main result} consists of three steps:

Step 1.
We construct for each $\rW$-orbit $\theta$ in $\calC(T)(F)$
a remarkable 
$\rW$-equivariant local system $\mE_\theta$ on $T$\footnote{The author learned the existence of $\mE_\theta$ from R. Bezruakvnikov} and 
consider the equivariant perverse sheaf or $D$-module $\cM_\theta=\Ind_{T\subset B}^G(\mE_\theta)^\rW$ on $G$.
A key observation is that to prove Theorem \ref{main result} it suffices 
to prove the acyclicity of $\cM_{\theta}$, that is, the cohomology vanishing properties~\eqref{desired vanishing} for
$\cM_{\theta}$.
This follows from a computation of the convolution of $\cM_\theta$
with $\Phi_\mF$ (Proposition \ref{conv with Phi}), where $\mF\in\sD_\rW(T)$ is a strongly $*$-central complex, 
and a result of Laumon on the conservativity of the Mellin transform (Lemma \ref{vanishing 1}).

Step 2. We use the techniques developed in \cite{BFO} on 
character $D$-modules and Drinfeld center of Harish-Chandra bimodules to prove
acyclicity of 
$\cM_{\theta}$ in the de Rham setting. 
An important point here is that 
$\mE_\theta$ is a tame local system on $T$, and hence the associated $D$-module 
$\cM_{\theta}$ on $G$ is a character $D$-module and the results in 
\emph{loc. cit.} are applicable.
A key step in the proof is to show that 
the Harish-Chandra bimodule corresponding to
$\mE_\theta$, under the Beilinson-Bernstein localization theorem, has a 
canonical central structure (Proposition \ref{E=Z}).

Step 3.
We construct a mixed characteristic lifting $\cM_{\theta,A}$ of $\cM_\theta$ over a strictly Henselian ring $A$ with residue field $k$ of characteristic not 
dividing $N\ell$, where $N$ is a positive integer depending only on the type 
of $G$.
We prove that $\cM_{\theta,A}$ is
universally locally acyclic with respect to the quotient map $\pi_A:G_A\to G_A/U_A$ (here $G_A$ and $U_A$ are 
models of $G$ and $U$ over $A$).
This allows us to deduce acyclicity of $\cM_\theta$ in the $\ell$-adic setting from the 
de Rham setting.
This completes the proof of Theorem \ref{main result}.

The results established in this paper reveal interesting connections between the
Braverman-Kazhdan conjectures, 
Whittaker $D$-modules,
character sheaves, and categorical center of Hecke categories. 
We plan to 
explore those connections in a future work \cite{C2}.

\begin{remark}
The proof of acyclicity of $\cM_\theta$ in the de Rham setting makes use of Harish-Chandra bimodules, and hence is algebraic. It would be interesting to have a geometric proof which treats the cases of various ground fields and
sheaf theories uniformly.
Presumably, such a proof will provide an explicit bound of the integer $N$ in 
Theorem \ref{main result}. 
 
\end{remark}

\begin{remark}
Using a similar but more involved argument, one can show that 
the vanishing conjecture holds for reductive groups with disconnected center.
Such a generalization is not needed for the proof of Braverman-Kazhdan conjectures, 
so the details will be worked out in a forthcoming paper \cite{C2}.

\end{remark}

\subsection{Organization}
We briefly summarize here the main goals of each section.
In Section 2 we collect standard notation in algebraic groups, $\ell$-adic 
sheaves, and $D$-modules.
In Section 3 we study induction and restriction functors.
In Section 4, we study characterizations of central and $*$-central 
complexes using Mellin transfroms.
In Section 5 we introduce the $*$-central local systems $\mE_\theta$ and 
establish some basic properties of them.
In Section 6 we prove acyclicity of character sheaves
$\cM_\theta$ (Theorem \ref{Key}). 
In Section 7 we prove Theorem \ref{main result}.

{\bf Acknowledgement.}
I especially like to thank
G\'erard Laumon for useful discussions. The argument 
using character sheaves in 
mixed-characteristic to prove the 
Braverman-Kazhdan conjectures in the $\ell$-adic setting, follows a suggestion of his.
I also would like to thank 
Roman Bezrukavnikov, Ng\^o Bao Ch\^au, Victor Ginzburg, and Zhiwei Yun for useful discussions. 
I am grateful for the support of NSF grant DMS-1702337.

\section{Notations}

\subsection{}
We denote by 
$\mB=G/B$ the flag variety. 
We denote by $\fg$, $\fb$, $\ft$, $\fn$ the Lie algebras of $G,B,T,U$.
We denote by $G^{\on{rs}}$ (resp. $T^{\on{rs}}$) the open subset consisting of regular semi-simple elements
in $G$ (resp. $T$). 
We denote by $G^{\on{reg}}$ the open subset consisting of regular elements in
$G$.
We denote by $\bG_a$ the additive group
and $\bG_m$ the multiplicative group.
We denote by $\check G$ the complex dual group of $G$ and 
$\check T$ the dual maximal torus.
We denote by 
$\rW_a^{\on{ex}}=\rW\ltimes\Lambda$ the extended affine Weyl group and 
$\rW_a=\rW\ltimes\rm R$ the affine Weyl group.
Here $\Lambda=\bX_\bullet(\check T)=\Hom(\bG_m,\check T)$ is the co-character lattice 
and $\rm R\subset\Lambda$
is the set of co-roots of $\check G$.
\subsection{}
For an algebraic stack $\calX$ over $k$, we denote by
$\sD(\calX)$ the bounded derived category of $\ell$-adic sheaves on 
$\calX$ in the $\ell$-adic setting or the bounded derived category of 
holonomic $D$-modules on $\calX$ in the de Rham setting.

For a smooth scheme $X$, we will write $\mathrm 1_X\in\sD(X)^\heartsuit$ for the constant perverse sheaf $\barQ[\dim X]$ on $X$ in the $\ell$-adic setting or the structure sheaf 
$\mO_X$ in the de Rham setting.

For a representable morphism $f:\calX\to\calY$,
the six functors $f^*,f_*,f_!,g^!,\otimes,\underline\Hom$ 
are understood in the derived sense. 
For a smooth map $f:\calX\to\calY$ of relative dimension $d$ we 
write $f^\circ=f^*[d]=f^![-d]$.

For an algebraic group $H$ over $k$ acting on a 
$k$-scheme $X$, we denote by $X/H$, the corresponding quotient stack
and $X//H$ the geometric invariant quotient (if exists).
We will write $H/_\ad H$ for the quotient stack of $H$ with respect to the 
adjoint action.
Consider the case when $H$ is a finite group.
Then the pull-back along the quotient map $X\to X/H$ induces an equivalence 
between 
$\sD(X/H)$ and the (naive) $H$-equivariant derived category 
on $X$, denoted by $\sD_H(X)$, whose objects consist of pair 
$(\mF,\phi)$, where $\mF\in\sD(X)$ and $\phi:a^*\mF\is\pr^*\mF$ is an isomorphism 
satisfying the usual compatibility conditions (here $a$ and $\pr$ are the action and projection map from $H\times X$ to $X$ respectively).\footnote{This holds in a more general situation when the neutral component of $H$ is unipotent.} 
We will call an object $(\mF,\phi)$ in $\sD_H(X)$ a $H$-equivariant complex and 
$\phi$ a $H$-equivariant structure on $\mF$. 
For simplicity, 
we will write $\mF=(\mF,\phi)$ for an object in $\sD_H(X)$.

We denote by
$\tau_{\leq n},\tau_{\geq n}$ the truncation functors corresponding to the 
standard $t$-structure on
$\sD(\calX)$. For any $\mF\in\sD(\calX)$, we denote by 
$\sH^n(\mF)$ the $n$-th cohomology sheaf.
In the $\ell$-adic setting,
we denote by 
$^p\tau_{\leq n},^p\tau_{\geq n}$ the truncation functors
corresponding to the perverse $t$-structure. 
For any $\mF\in \sD(\calX)$, the $n$-th perverse cohomology sheaf is defined as 
$^p\mathscr H^n(\mF)={^p}\tau_{\geq n}{^p}\tau_{\leq n}(\mF)[n]$.

Depending on the setting, we write 
$\sD(\calX)^\heartsuit$ for the heart corresponding to the 
perverse $t$-structure 
in the $\ell$-adic setting 
and the heart corresponding to the 
standard $t$-structure in the de Rham setting.

For any stack $\calX$ over $k$, we denote by
$\on{Coh}(\calX)$ and $\on{QCoh}(\calX)$ the categories of 
coherent and quasi-coherent sheaves on $\calX$, and 
$D_{coh}^b(\calX)$ and $D_{qcoh}^b(\calX)$ the corresponding 
bounded derived categories.

Let $\mF$ be an
quasi-coherent sheaf or 
$D$-module on a scheme. We will write 
$\Gamma(\mF)$ and $R\Gamma(\mF)$ for the global section and derived global section
of $\mF$ as an quasi-coherent sheaf.
For any scheme $X$
we will write $\mO_X$ for the structure sheaf of $X$ and 
$\mO(X)=\Gamma(\mO_X)$ the ring of global functions on $X$.

Assume $k=\bC$. For any smooth scheme $X$ we denote by 
$\mD_X$ the sheaf of differential operators on $X$.
Let
$f:X\ra Y$ be a principal $T$-bundle over a smooth scheme $Y$.
A $D$-module $\cF$ on $X$ is called $T$-monodromic
if it is weakly $T$-equivariant (see \cite[Section 2.5]{BB}). 
A object $\cF\in\sD(X)$ is called $T$-monodromic if 
$\sH^i(\cF)$ is $T$-monodromic for all $i$. Let $\cF\in\sD(X)$
be a $T$-monodromic object.
 For any $\mu\in\check\ft\is\ft^*$, we denote $\Gamma^{\hat\mu}(\cF)$ 
(resp. $R\Gamma^{\hat\mu}(\cF)$) the maximal summand 
of $\Gamma(\cF)$ (resp. $R\Gamma(\cF)$)
where $\ft$, acting as infinitesimal
translations along the action of $T$, acts with the generalized eigenvalue $\mu$.

\section{Induction and restriction functors}
In this section we collect some known facts about 
induction and restriction functors.

\subsection{}
Recall the Grothendieck-Springer simultaneous resolution 
of the Steinberg map $c:G\to T//W$:
\beq\label{G-S resolution}
\xymatrix{\widetilde G\ar[r]^{\tilde q}\ar[d]^{\tilde c\ \ }&T\ar[d]^q\\
G\ar[r]^c&T//\rW}
\eeq
where $\widetilde G$ is the closed subvariety of $G\times G/B$ 
consisting of 
pairs $(g,xB)$ such that $x^{-1}gx\in B$.
The map $\tilde c$ is given by $(g,xB)\to g$, and the map 
$\tilde q$ is given by $(g,xB)\to x^{-1}gx\on{\ mod} U\in B/U=T$.  
The induction functor $\Ind_{T\subset B}^G:\sD(T)\to\sD(G)$
is given by \[\Ind_{T\subset B}^G(\mF)=
\tilde c_*\tilde q^\circ(\mF).\]
We have the following equivalent constructions of 
$\Ind_{T\subset B}^G$. Consider the fiber product 
\[Z=G\times_{T//\rW}T.\]
It is known that the map $h:\tilde G\to Z$ induced from~\eqref{G-S resolution} is a small map and it follows that 
$\IC(Z)=h_!\mathrm 1_S\in\sD(Z)^\heartsuit$ is the IC-complex of $Z$.
We have 
\beq\label{Ind IC}
\Ind^G_{T\subset B}(\mF)\is (p_G)_*(p_T^*(\mF)\otimes\IC(Z))[\dim G-\dim T]\eeq 
where $p_T:Z\to T$ and $p_G:Z\to G$ are the natural projection map.
Let $X$ be a scheme acted on by an algebraic group $H$ and let
$H'\subset H$ be a subgroup.  
There is an averaging functor
$\Av_{*}^{H/H'}=\pi_*:\sD(X/H')\to\sD(X/H)$ (resp. $\Av_{!}^{H/H'}=\pi_!:\sD(X/H')\to\sD(X/H)$)
which is the right adjoint (resp. left adjoint) of the 
forgetful functor $\on{oblv}^{H/H'}:\sD(X/H)\to\sD(X/H')$. 
Here $\pi:X/H'\to X/H$ is the natural map.
If $H'=e$ is trivial we will omit $H'$ and 
simply write 
$\on{Av}_*^H=\on{Av}_*^{H/e}$ (resp. $\on{Av}_!^H=\on{Av}_*^{H/e}$) 
and $\on{oblv}^H=\on{oblv}^{H/e}$.

Note that, for any $\mF\in\sD(T)$, its $*$-pull back along  $B\to T=B/U$, denoted by
$\mF_B$, can be regarded as an object in $\sD(G/_\ad B)$ and there is a canonical isomorphism 
\[\on{oblv}^{G/e}\circ\Av_*^{G/B}(\mF_B)\is\Ind_{T\subset B}^G(\mF).\]

The functor $\Ind_{T\subset B}^G$ admits a righ adjoint
$\Res_{T\subset B}^G:\sD(G)\to \sD(T)$, called the restriction functor, and is given by
\[\Res_{T\subset B}^G(\mF)= (q_B)_*i_B^!(\mF) \]
where $i_B:B\to G$ is the natural inclusion and 
$q_B:B\to T=B/U$ is the quotient map.
More generally, one could define 
$\Res_{L\subset P}^G: \sD(G)\to \sD(L)$, for any 
pair $(L,P)$ where $L$ is a Levi subgroup of a parabolic subgroup $P$ of $G$.

We have the following exactness properties of 
induction and restriction functors:
\begin{proposition}\label{exactness}
(1) The functor $\Ind_{T\subset B}^G$ maps perverse sheaves on $T$ to 
perverse sheaves on $G$.
(2) The functor 
$\Res_{L\subset P}^G$ maps 
$G$-conjugation equivariant perverse sheaves on $G$ to 
$L$-conjugation equivariant perverse sheaves on $L$.
\end{proposition}
\begin{proof}
This is 
\cite[Theorem 5.4]{BY}.
\end{proof}
\subsection{$\rW$-action}
Let $\mF\in\sD_\rW(T)$.
Since the map $p_T:S\to T$ and the 
$\IC$-complex $\IC(S)$ are
 $\rW$-equivariant, it follows from~\eqref{Ind IC} that the $\rW$-equivariant structure on
 $\mF$ gives rise to a $\rW$-action on 
 $\Ind_{T\subset B}^G(\mF)$.
 We denote by 
 \beq\label{inv factor} 
 \Phi_\mF:=\Ind_{T\subset B}^G(\mF)^\rW
 \eeq
 the $\rW$-invariant factor of $\Ind_{T\subset B}^G(\mF)$.

In the case when $\mF$ is a $\rW$-equivariant perverse local system on $T$,
we have the following description of $\Phi_\mF$:
Let $q^{\rs}:T^{\rs}\to T^{\rs}//\rW$
and $c^{\rs}:G\to T^{\rs}//\rW$ be the restriction of 
the maps in~\eqref{G-S resolution} to the regular semi-simple locus.
As $q^{\rs}$ is an \'etale covering, 
the restriction of 
$\mF$ to $T^{\rs}$ descends to a perverse local system $\mF'$ on 
$T^{\rs}//\rW$ and we have 
\[\Phi_\mF=\Ind_{T\subset B}^G(\mF)^\rW\is j_{!*}(c^{\rs})^*\mF'[\dim G-\dim T].\]

Let $\mF\in\sD_\rW(T)^\heartsuit$
and let $F[\rW]\to\End_{\sD(G)^\heartsuit}(\Ind_{T\subset B}^G(\mF))$ be the map coming from
the $\rW$-action.
By adjunction, we get a map
\[F[\rW]\to\End_{\sD(G)^\heartsuit}(\Ind_{T\subset B}^G(\mF))\is\Hom_{\sD(G)^\heartsuit}(\mF,\Res_{T\subset B}^G\circ\Ind_{T\subset B}^G(\mF)),\]
which gives rise to
\beq\label{W-action}
F[\rW]\otimes\mF\to\Res_{T\subset B}^G\circ\Ind_{T\subset B}^G(\mF).
\eeq

We have the following generalization of 
\cite[Theorem 4.6]{Gu} to the group setting:

\begin{proposition}\label{Mackey formula}
Let $\mF\in\sD(T)^\heartsuit$. 
(1) There is a canonical isomorphism 
$\bigoplus_{w\in\rW} w^*\mF\is\Res_{T\subset B}^G\circ\Ind_{T\subset B}^G(\mF)$.
(2) Assume further that $\mF\in\sD_\rW(T)^\heartsuit$.
Then the composition 
\beq\label{W-action 2}
F[\rW]\otimes\mF\stackrel{\phi}\is\bigoplus_{w\in\rW} w^*\mF\stackrel{}\is\Res_{T\subset B}^G\circ\Ind_{T\subset B}^G(\mF),
\eeq
is equal to~\eqref{W-action}.
Here $\phi$ is the $\rW$-equivariant structure of $\mF$.
\end{proposition}
\begin{proof}
We follow the argument in \emph{loc. cit.}. We shall give a proof in the de Rham setting. 
The same proof works for the $\ell$-adic setting. 
Consider the product $S_G=\widetilde G\times_GB$. 
There are two natural maps $q_S:S_G\to T$ and $c_S:S_G\to T$
coming from $\tilde q:\tilde G\to T$ and $\tilde c:\tilde G\to T$ in~\eqref{G-S resolution}.
Let $\mS_G:=\sH^0(q_S\times c_S)_*(\omega_{S_G})$.
Using base changes formulas, it is easy to see that the functor 
 $\Res_{T\subset B}^G\circ\Ind_{T\subset B}^G(-)$ is given by the 
 kernel $\mS_G$, that is, we have,
$\Res_{T\subset B}^G\circ\Ind_{T\subset B}^G(\mF)\is \pr_{l,*}(\pr_r^\circ\mF\otimes\mS_G)$, here 
$\pr_l:\pr_r:T\times T\to T$ are the left and right projection maps.
Now the same proof as in \cite[Section 4.3]{Gu}, replacing 
$\ft$ by $T$, shows that there is a canonical isomorphism of monads
$\mS_G\is\bigoplus_{w\in\rW}\mO_{\Gamma_w}$, here 
$\Gamma_w=\{x,wx|x\in T\}\subset T\times T$. 
It follows that 
$\Res_{T\subset B}^G\circ\Ind_{T\subset B}^G(\mF)\is \pr_{l,*}(\pr_r^\circ\mF\otimes\mS_G)\is\bigoplus_{w\in\rW} w^*\mF$.
This completes the proof of (1).
Assume $\mF$ is $\rW$-equivariant. It follows from the construction that the isomorphism in (1) intertwines the $\rW$-action on 
$\Res_{T\subset B}^G\circ\Ind_{T\subset B}^G(\mF)$ with 
the one on $\bigoplus_{w\in\rW} w^*\mF$ given by the map 
\[a_{r}:\bigoplus_{w\in\rW} w^*\mF\stackrel{w^*(\phi_r)}\to \bigoplus_{w\in\rW} w^*r^*\mF\is\bigoplus_{w\in\rW} (rw)^*\mF=\bigoplus_{w\in\rW} w^*\mF,\] 
for any $r\in\rW$.
Part (2) follows.
\end{proof}

We will need the following properties of induction functors.
Let $m_G:G\times G\to G$ and $m_T:T\times T\to T$
be the multiplication maps.
For any $\mM,\mM'\in\sD(G)$, we define 
$\mM*\mM':=m_{G,*}(\mM\boxtimes\mM')\in\sD(G)$.
Similarly, for any $\mF,\mF'\in\sD(T)$, we define
$\mF*\mF'=m_{T,*}(\mF\boxtimes\mF')$.

\begin{prop}\label{prop of ind}
Let $\mM\in\sD(G/_\ad G)^\heartsuit$ and $\mF\in\sD(T)$. 
Assume
$\on{Av}^U_*(\mM)$ is supported on $T=B/U\subset G/U$.
Then we have a natural isomorphism in $\sD(G)$
\[\Ind_{T\subset B}^G(\mF)*\mM\is\Ind_{T\subset B}^G(\mF*\Res_{T\subset B}^G\mM)\]
which is 
functorial with respect to $\mF$.

\end{prop}
\begin{proof}
This is proved in \cite[Proposition 2.9]{BK2}. 
Let us recall the construction in \emph{loc.cit.}.
For any $\mH\in\sD(G/_\ad B)$ and $\mM\in\sD(G/_\ad G)^\heartsuit$, there is a natural isomorphism
\[\Av_*^{G/B}(\mH*\on{oblv}^{G/B}\mM)\is\Av_*^{G/B}(\mH)*\mM.\]
When $\mH=\mF_B=q_B^*\mF$ is the pull back of $\mF$ along $q_B:B\to T=B/U$,
the assumption that $\Av_*^U(\cM)$ is supported on $T$ implies that
\[\mF_B*\on{oblv}^{G/B}\cM\is(\mF*\Res_{T\subset B}^G(\cM))_B\in\sD(G/_\ad B)\] and
it follows that 
\[\Ind_{T\subset B}^G(\mF)*\cM\is
\on{oblv}^{G/e}\circ(\Av_*^{G/B}(\mF_B)*\cM)\is
\on{oblv}^{G/e}\circ(\Av_*^{G/B}(\mF_B*\on{oblv}^{G/B}\cM))\is\]
\[\is\on{oblv}^{G/e}\circ(\Av_*^{G/B}((\mF*\Res_{T\subset B}^G(\cM))_B)\is\Ind_{T\subset B}^G(\mF*\Res_{T\subset B}^G(\cM)).\]
\end{proof}

\section{Mellin transform and characterizations of central 
and $*$-central complexes}

\subsection{The scheme of tame characters}\label{C(T)}
Let $\pi^t_1(T)$ be the tame \'etale fundamental group of $T$ 
in the $\ell$-adic setting or the topological fundamental group
in the de Rham setting.
For any continuous character character $\chi:\pi^t_1(T)\to F^\times$, 
we denote by $\mL_\chi$  the corresponding rank one 
local system on $T$.  

We first 
consider the de Rham setting. Set $\calC(T)=\Spec(\bC[\pi_1^t(T)])$. 
Then the $\bC$-points of $\calC(T)$ are in bijection with 
characters of $\pi_1^t(T)$. Note that, under the isomorphism $\pi_1^t(T)=\mathbb X_\bullet(T)$, characters of $\pi_1^t(T)$ 
 correspond to elements in the dual torus $\check T$.

In the $\ell$-adic setting,
in \cite{GL}, a $\barQ$-scheme $\calC(T)$ is defined, whose 
$\barQ$-points are in bijection with 
continuous characters of $\pi^t_1(T)$. There is decomposition 
\[\calC(T)=\bigsqcup_{\chi\in\calC(T)_f}\{\chi\}\times\calC(T)_\ell\] into connected components, where 
$\calC(T)_f\subset\calC(T)$ is the subset consisting of tame characters of 
order prime to $\ell$ and 
$\calC(T)_\ell$ is the connected component of $\calC(T)$ containing 
the trivial character.
It is shown in \emph{loc. cit.} that $\calC(T)$ is Noetherian and regular and there is an isomorphism
\[\calC(T)_\ell\is\Spec(\barQ\otimes_{\bZ_\ell}\bZ_\ell[[x_1,...,x_r]]).\] 
In addition, the $\barQ$-points of $\calC(T)_\ell$ are in bijection 
with pro-$\ell$ characters of $\pi_1(T)$ (i.e. characters of 
the pro-$\ell$ quotient $\pi_1(T)_\ell$ of $\pi_1(T)^t$).

\subsection{Mellin transforms}\label{MT}
We give a review of Mellin transforms in both de Rham and $\ell$-adic settings
and establish some basic facts about them.

We first recall the Mellin transform of $D$-modules on 
$T$.
Let $x_i\in\Lambda=\bX_\bullet(\check T)\is\Hom(T,\bG_m)$ be a basis and consider the 
regular function $\mO(T)\is\bC[x_i^{\pm1}]$ and the algebra 
of differential operators 
$\Gamma(\mD_T)\is\bC[x_i^{\pm1}]\langle v_i\rangle/\{v_ix_j=x_j(\delta_{ij}+v_i)\}$
where $v_i=x_i\partial_{x_i}\in\ft$ are a basis for the $T$-invariant 
vector fields.
Recall that for any $D$-module $\mF$ on $T$, the tensor product
$\mF\otimes_{}\omega_T$ with the canonical line bundle $\omega_T$ on $T$, carries a natural right $D$-module structure.
Note also that, if   
we consider $\Gamma(\mD_T)$ as the algebra of difference 
operators $\bC[v_i]\langle x_i^{\pm1}\rangle/\{v_ix_j=x_j(\delta_{ij}+v_i)\}$, then
there is 
a canonical equivalence between 
the category $\on{QCoh}(\check\ft/\Lambda)$ of $\Lambda$-equivariant quasi-coherent
sheaves on $\check\ft$ and the category of right $\Gamma(\mD_T)$-modules\footnote{For any right $\Gamma(\mD_T)$-module $\calN$, the action of $\bC[v_i]=\mO(\check\ft)$ on 
$\calN$
gives a $\mO_{\check\ft}$-module structure on $\mO_{\check\ft}\otimes_{\mO(\check\ft)}\calN\in\on{QCoh}(\check\ft)$ and the action of $x_i\in\Lambda$ defines a $\Lambda$-equivariant structure.}.  
The Mellin transform functor is defined as 
\[\frak M:D\on{-mod}(T)\ra\on{QCoh}(\check\ft/\Lambda),\ \ \calN\to\Gamma(\cN\otimes\omega_T).\]
We have the following properties:
\begin{enumerate}
\item
The functor $\frak M$ is 
an equivalence.
\item 
Let $\chi\in\check T(\bC)\is\calC(T)(\bC)$ and let $\lambda\in\check\ft(\bC)$ be a lift of 
$\chi$ along the universal covering $\exp:\check\ft\to\check T$. 
Then for any $\mF\in D\on{-mod}(T)$ we have 
\beq\label{stalk}
\oH^*(T,\mF\otimes\mL_\chi)\is i_\lambda^*\mathfrak M(\mF)
\eeq
here $i_\lambda:\on{pt}\to\check\ft$ is the embedding given by $\lambda$.
\item
Consider the bounded derived category 
$D^b_{qcoh}(\check\ft/\Lambda)$
of $\Lambda$-equivariant 
quasi-coherent sheaves on $\check\ft$
with the monoidal 
structure given by the (derived) tensor product. 
We have 
\[\mathfrak M(\cF*\cF')\is\mathfrak M(\cF)\otimes\mathfrak M(\cF').\]  

\item
Let $\mF$ be a $\rW$-equivariant $D$-module $\mF$ on $T$. Then $\rW$-equivariant structure on
$\mF$ gives rise to a $\rW_a^{\on{ex}}=\rW\ltimes\Lambda$-equivariant structure on
$\frak M(\mF)$.
Let $\chi$ and $\lambda$ be as in $(1)$ and 
let $\rW_{a,\lambda}^{\on{ex}}$ and $\rW_{a,v}$ be the stabilizers of $\lambda$ in
$\rW_a^{\on{ex}}$ and $\rW_a$ respectively.
Then $\oH^*(T,\mF\otimes\mL)$ (resp. $i_\lambda^*\frak M(\mF)$)
carries a natural action of $\rW_\chi'$ (resp. $\rW_{a,\lambda}^{\on{ex}}$) such that,
under the isomorphism 
\beq\label{iso of stabilizers}
\rW_\chi'\is\rW_{a,\lambda}^{\on{ex}},\ \ 
w\to (w,w^{-1}\lambda-\lambda)
\eeq
the isomorphism~\eqref{stalk} intertwines those actions.

\end{enumerate}

We now consider the $\ell$-adic setting.
In \cite{GL}, the authors constructed the Mellin transform
\[\frak M:\sD(T)\to D^b_{coh}(\calC(T))\footnote{It is denoted by $\sM_*$ in \emph{loc.cit.}}\]
with the following properties:
\begin{enumerate}
\item
Let $\chi\in\calC(T)(\barQ)$ and $i_\chi:\on{pt}\to\calC(T)$ be the embedding given by 
$\chi$. We have 
\[\oH^*(T,\mF\otimes\mL_\chi)\is i^*_\chi\mathfrak M(\mF).\]

\item For any $\chi\in\calC(T)(\barQ)$ we have 
\[\frak M(\mF\otimes\mL_\chi)\is
m_\chi^*\frak M(\mF).\]
\item 
The functor $\frak M$ is t-exact with respect to the perverse $t$-structure on 
$\sD(T)$ and the standard $t$-structure on $D^b_{coh}(\calC(T))$.
Moreover, for any $\mF\in \sD(T)$, $\mF$ is perverse if and only if 
$\frak M(\mF)$ is a coherent complex in degree zero.

\item 
We have
\[\frakM(\mF* \mF')\is\frakM(\mF)\otimes\frakM(\mF')\]

\item The Mellin transforms restricts to an equivalence 
\beq\label{inverse Mellin}
\mathfrak M:\sD(T)_{\on{mon}}\is D^b_{coh}(\calC(T))_f
\eeq
between the full subcategory $\sD(T)_{\on{mon}}$ of monodromic $\ell$-adic complexes on $T$ and 
the full subcategory $D^b_{coh}(\calC(T))_f$ of coherent complexes on $\calC(T)$ with 
finite support.

\item
For $\mF\in\sD_\rW(T)$, 
the $\rW$-equivariant structure on $\mF$ gives rise to a 
$\rW$-equivariant structure on
$\mathfrak M(\mF)$ such that
for any $\chi\in\calC(T)(\barQ)$
the isomorphism 
in (1) above 
is compatible with the natural 
$\rW'_\chi$-actions. 
\end{enumerate}

\begin{remark}\label{global section}
In the de Rham setting,
a non-zero invariant section $\sigma$ of $\omega_T$ gives rise to 
an isomorphism 
$\Gamma(\mathfrak M(\mF))=\Gamma(\mF\otimes\omega_T)\stackrel{\sigma}\is\Gamma(\mF)$,
here $v_i$ acts on $\Gamma(\mF)$ by the formula 
$v_i\cdot m=-v_im$.
Since the $\rW$-action on invariant sections of of $\omega_T$ is given by the sign character, 
we obtain an isomorphism of $\mO(\check\ft)$-modules
\[\Gamma((-1)^{*}\mathfrak M(\mF\otimes\sign)))\is\Gamma(\mF)\]
compatible with the natural $\rW_a^{\text{ex}}$-actions. Here $-1:\check\ft\to\check\ft, x\to -x$.
\end{remark}

The properties of Mellin transforms above imply the following:
\begin{lemma}\label{chara of central}
Let $\mF\in\sD_\rW(T)$. 
(1) Assume $\on{char}k>0$. Then $\mF$ is $*$-central (resp. strongly $*$-central) if and only if 
, for any $\chi\in\calC(T)(\barQ)$, the action of 
$\rW_\chi$ (resp. $\rW_\chi'$) on $\sH^n(i_\chi^*\frak M(\mF\otimes\sign))$
for all $n\in\in\bZ$ is trivial.
(2) Assume $k=\bC$. Then $\mF$ is $*$-central (resp. strongly $*$-central) if and only if,
for any $\lambda\in\check\ft(\bC)$,
 the action of 
$\rW_{a,\lambda}$ (resp. $\rW_{a,\lambda}^{\on{ex}}$) on $\sH^n(i_\lambda^*\frak M(\mF\otimes\sign))$ 
for all $n\in\bZ$
is trivial.

\end{lemma}

We finish this section with a lemma to be used in Section \ref{Proof}.
\begin{lemma}\label{finiteness}
Assume $k=\bC$. Let $\lambda\in\check\ft(\bC)$
and let $\check\ft_{\hat\lambda}$ be the completion of $\check\ft$ at $\lambda$.
For any holonomic complex $\mF\in\sD(T)$, the restriction 
$\frak M(\mF)|_{\check\ft_{\hat\lambda}}$ is a coherent complex on 
$\check\ft_{\hat\lambda}$\footnote{Since $\frak M(\mF)$ is in general not coherent, the claim is not automatic.}.
\end{lemma}
\begin{proof}
By induction argument on the (finite) number of non vanishing cohomology sheaves of $\mF$, we can assume $\mF\in\sD(T)^{\heartsuit}$. 
Let $I$ be the maximal ideal corresponding to $
\lambda$ and let $R=\mO(\check\ft)$, $M=\Gamma(\frak M(\mF))$.
We shall show that the natural injection 
$M\otimes_R R_{\hat\lambda}=M\otimes_R\underleftarrow{\lim} R/I^kR\to M_{\hat\lambda}=\underleftarrow{\lim} M/I^kM $
is an isomorphism and $M_{\hat\lambda}$ is a finitely generated 
$R_{\hat\lambda}$-module.   
Note that, since $\mF$ is holonomic, the non-derived fiber 
$\sH^0(i_\lambda^*\frak M(\mF))\is M/IM$ is finite dimensional.
Choose a finitely generated submodule $N\subset M$ such that natural map
$N/I_\lambda N\to M/I_\lambda M$ is onto.
This implies $N/I_\lambda^nN\to M/I_\lambda^nM$ is onto for all $n>0$.
Indeed, let $N'$ be the image of the map above. 
Then we have $M/I^nM=N'+I M/I^nM$.
By iterating the equality above,
we obtain
$M/I^nM=N'+I (N'+I M/I^nM)=
N'+I^2 (M/I^nM)=\cdot\cdot\cdot =N'+I^k (M/I^nM)$ for any $k>0$, and take $k=n$, we get 
$M/I^nM=N'$.
We have shown that $N_{\hat\lambda}\to M_{\hat\lambda}$ is surjective and, as $N$ is finitely generated, it follows that
$M_{\hat\lambda}$ is finitely generated and
$M\otimes_RR_{\hat\lambda}\to M_{\hat\lambda}$ is an isomorphism.
 
 \end{proof}

\subsection{Characterization of $*$-central $D$-modules}\label{characterization}
In the de Rham setting,
we have the following characterization of $*$-central $D$-modules:
\begin{thm}\label{characterization of central}
Let $\mF$ be a $\rW$-equivariant holonomic $D$-module on $T$. 
The following are equivalent
\begin{enumerate}
\item $\mF$ is $*$-central.
\item  for any $\lambda\in\check\ft(\bC)$,
 the action of 
$\rW_{a,\lambda}$ on $\sH^n(i_\lambda^*\frak M(\mF\otimes\sign))$, $n=0,-1$, is 
trivial.
\item
for any $\lambda\in\check\ft(\bC)$, the Mellin transform $\frak M(\mF\otimes\sign)$, regarding as a $\rW_{a,\lambda}$-equivaraint quasi-coherent sheaf
on $\check\ft$,
descends to $\check\ft//\rW_{a,\lambda}$.
\item
for any $\lambda\in\check\ft(\bC)$, the natural map
$\mO(\check\ft)\otimes_{\mO(\check\ft)^{\rW_{a,\lambda}}}\Gamma(\mF)^{\rW_{a,\lambda}}\to\Gamma(\mF)$
is a bijection.

\end{enumerate}
\quash{
$\mF$ is $*$-central
\item  for any $\chi\in\calC(T)(\barQ)$, the action of 
$\rW_\chi$ on $\sH^{n}(i_\chi^*\frak M(\mF\otimes\sign))$, $n=0,-1$, is trivial.
\item for any $\chi\in\calC(T)(\barQ)$, the Mellin transform $\frak M(\mF\otimes\sign)$, regarding as a $\rW_\chi$-equivaraint coherent sheaf
on $\calC(T)$,
descends to $\calC(T)//\rW_\chi$.
\end{enumerate}}
\end{thm}

A similar result in the $\ell$-adic setting was proved in \cite[Proposition 4.2]{C1}.

\begin{remark}
Note that in the de Rham setting the definition of 
$*$-central $D$-modules makes sense for arbitrary $\rW$-equivaraint $D$-modules on $T$,
and the proof of Theorem \ref{characterization of central} below shows that 
the above characterization 
remains true without the holonomicity assumption on
$\mF$.
Very similar results are proved in \cite{Gi2,L}.
 
\end{remark}

We begin will the following lemma.
\begin{lemma}\label{key descent}
Let $\Gamma$ be a finite reflection group with reflection representation 
$V$ over $\bC$. Let $\mF$ be a $\Gamma$-equivariant quasi-coherent sheaf on
$V$. Then $\mF$ descends to $V//\Gamma$ if and only if 
for any $\lambda\in V(\bC)$ the actions of the stabilizer $\Gamma_\lambda$ of 
$\lambda$ in $\Gamma$ on $\sH^0(i_\lambda^*\mF)$ and $\sH^{-1}(i_\lambda^*\mF)$ are trivial.
Here $i_\lambda:\on{pt}\to V$ is the embedding given by $\lambda$.

\end{lemma}
\begin{proof}
Assume $\mF$ descends to $V//\Gamma$.
Then we have 
$\mF\is \pi^*(\pi_*\mF)^\Gamma$ where $\pi:V\to V//\Gamma$ is the quotient map
and 
it implies 
$i_\lambda^*\mF\is i_{\bar\lambda}^*(\pi_*\mF)^\Gamma$, where 
$\bar\lambda=\pi(\lambda)$ and
$i_{\bar\lambda}:\on{pt}\to V//\Gamma$ is the inclusion.
As $\Gamma$ acts trivially on $\sH^n(i_{\bar\lambda}^*(\pi_*\mF)^\Gamma)$ for all $n\in\bZ$, 
it follows that $\Gamma$ acts trivially on $\sH^n(i_\lambda^*\mF)$ for all $n\in\bZ$, in particular, for 
$n=0,-1$.

Assume $\Gamma$ acts trivially on $\sH^n(i_\lambda^*\mF)$ for $n=0,-1$.
We would like to show that 
$\mF$ descends to $V//\Gamma$.
By \cite[Theorem 1.3.2]{L}, it suffices to 
show that $\mF$ descends to $V//\langle\sigma\rangle$ for any simple reflection $\sigma\in\Gamma$.
So we could assume $\Gamma=\langle\sigma\rangle$ is generated by a reflection $\sigma$.
Let $\mF^{\sigma}=\{h\in\mF|\sigma(h)=h\}$ and
$\mF^{\sigma=-1}=\{h\in\mF|\sigma(h)=-h\}$.
Choose a coordinate $(x_1,...,x_{n})$  of $V$ such that 
$\sigma(x_1)=-x_1$ and $\sigma(x_i)=x_i$ for $i\geq 2$.
Let $\lambda\in V(\bC)$ be the origin with coordinate 
$x_i=0$.

We claim that $\Gamma=\Gamma_\lambda$ acts trivially on $\sH^{0}(i_{\lambda}^*\mF)=\mF/(x_1,...,x_n)\mF$
implies the natural map 
\beq\label{map}
\mO(V)\otimes_\bC\mF^\sigma\to\mF
\eeq is surjective.

Step 1. We first show that 
$\Gamma$ acts trivially on $\mF/(x_1)\mF$ implies 
~\eqref{map} is surjective.
Indeed, the assumption implies that
the image of $f\in\mF^{\sigma=-1}$ in the quotient $\mF/(x_1)\mF$ is zero, that is, $f\in (x_1)\mF$.
Since $\sigma(x_1)=-x_1$ and
$\mF=\mF^\sigma\oplus\mF^{\sigma=-1}$, it follows that 
$f=x_1f'$ for some $f'\in\mF^\sigma$ and it implies~\eqref{map} is surjective.

Step 2.
We show that
$\Gamma$ acts trivially on
$\sH^{0}(i_{\lambda}^*\mF)$ implies 
$\Gamma$ acts trivially on $\mF/(x_1)\mF$.
The case $n=1$ is trivial as $\sH^{0}(i_{\lambda}^*\mF)=\mF/(x_1)\mF$.
Assume $n>1$.
Consider the exact sequence 
\beq\label{coker}
\mO(V)\otimes_\bC\mF^\sigma\to\mF\to\mM\to0
\eeq
where $\mM$ is the cokernel.
The quotient $\mF'=\mF/(x_2,...,x_n)\mF$ 
is a $\Gamma$-equivariant quasi-coherent sheaf on 
$V'=\on{Spec}(\bC[x_1])$ such that $\Gamma$ acts trivially on
$\mF'/(x_1)\mF'=\mF/(x_1,...,x_n)\mF$, thus Step 1 implies 
$\mO(V')\otimes(\mF')^\sigma\to\mF'$ is surjective. 
It follows that the first arrow in the exact sequence  
\[\mO(V)/(x_2,...,x_n)\mO(V)\otimes_\bC\mF^\sigma\to\mF/(x_2,...,x_n)\mF\to\mM/(x_2,...,x_n)\mM\to0\] 
 induced from~\eqref{coker} is surjective, and hence $\mM/(x_2,...,x_n)\mM=0$.
It implies $\mM/(x_1)\mM=\mM/(x_1,x_2,...,x_n)\mM$, which is an quotient of 
$\sH^{0}(i_{\lambda}^*\mF)=\mF/(x_1,...,x_n)\mF$, and thus $\Gamma$ acts trivially 
on $\mM/(x_1)\mM$. Consider the exact sequence 
\[\mO(V)/(x_1)\mO(V)\otimes_\bC\mF^\sigma\to\mF/(x_1)\mF\to\mM/(x_1)\mM\to0\]
induced from~\eqref{coker}.
As $\Gamma$ acts trivially on 
$\mO(V)/(x_1)\mO(V)\otimes_\bC\mF^\sigma$ and $\mM/(x_1)\mM$, it follows that
$\Gamma$ acts trivially on $\mF/(x_1)\mF$. 
This completes the proof of 
Step 2. The desired claim follows.

Consider the short exact sequence
\[0\to\mN\to\mO(V)\otimes_\bC\mF^\sigma\to\mF\to 0,\]
where $\mN$ is the kernel of~\eqref{map}.
It gives rise to an exact sequence
\beq\label{exact sequence}
\sH^{-1}(i_{\lambda}^*\mF)\to\mN/(x_1,...,x_n)\mN\to\mO(V)/(x_1,...,x_n)\mO(V)\otimes_\bC\mF^\sigma\to
\mF/(x_1,...,x_n)\mF\to 0.
\eeq
Since $\Gamma$ acts trivially on $\sH^{-1}(i_{\lambda}^*\mF)$ (by assumption) and $\mO(V)/(x_1,...,x_n)\mO(V)\otimes_\bC\mF^\sigma$, ~\eqref{exact sequence} implies that $\Gamma$ acts trivially on $\mN/(x_1,...,x_n)\mN$, and the claim above implies that the natural map $\mO(V)\otimes_\bC\mN^\sigma\to\mN$ is surjective.
All together, we obtain a $\Gamma$-equivariant presentation 
\[\mO(V)\otimes_\bC\mN^\sigma\to \mO(V)\otimes_\bC\mF^\sigma\to\mF\to0\]
where $\mO(V)\otimes_\bC\mN^\sigma$ and $\mO(V)\otimes_\bC\mF^\sigma$
are free $\Gamma$-equivaraint sheaves that descends to $V//\Gamma$ and,
by \cite[Lemma 3.1]{Ne}, it implies $\mF$ descends to $V//\Gamma$.

\end{proof}

\begin{remark}
The lemma above is inspired by \cite[Theorem 1.2]{Ne}.
In \emph{loc. cit.}, the author proved a similar descent criterion in the case when
$\Gamma$ is a reductive 
algebraic group and $\mF$ is a $\Gamma$-equivariant coherent sheaf.
The argument in \emph{loc. cit.} used the coherence assumption of $\mF$ hence can not be applied directly to 
the case of quasi-coherent sheaves. The proof above uses some special features of 
finite reflection groups.

\end{remark}

\subsubsection*{Proof of Theorem \ref{characterization of central}}
(1) implies (2) is clear. (2) implies (3) follows from the fact that 
$\Gamma=\rW_{a,v}$ is generated by 
affine reflections passing through $\lambda\in\check\ft(\bC)$ and, for any $\mu\in\check\ft(\bC)$, the stabilizer
$\Gamma_\mu=\rW_{a,\lambda}\cap\rW_{a,\mu}^{\on{ex}}$ is a subgroup of 
$\rW_{a,\mu}$, hence acts trivially on 
$\sH^j(i_\mu^*\frak M(\mF\otimes\sign))$, $j=0,-1$.
(3) implies (1) follows from the first paragraph of the proof of Lemma \ref{key descent}.
The equivalence between (3) and (4) follows from Remark \ref{global section}
and the fact $\mathscr M(\mF\otimes\sign)$ descends to $\check\ft//\rW_{a,v}$
if and only if $(-1)^*\mathscr M(\mF\otimes\sign)$ descends to $\check\ft//\rW_{a,-v}$.

\subsection{ 
Whittaker sheaves on $G$}\label{Whittaker}
Consider the de Rham setting.
Fix a generic character $f:\fn\to\bC$. 
A holonomic $D$-modules $\mF$ on $G$ is called a 
Whittaker $D$-module if for any vector $n\in\fn$, 
the actions of $a_l(n)-f(n)$ and $a_r(n)-f(n)$ on $\Gamma(\mF)$ are locally nilpotent.
Here $a_l$ (resp. $a_r$) is the embedding of $\fg$ in $\Gamma(\mD_G)$ as 
left-invariant (resp. right-invaraint) differential operators.
In \cite[Theorem 1.5.1]{Gi2} and \cite[Theorem 1.2.2]{L}, Ginzburg and Lonergan proved 
that the category of Whittaker $D$-modules on $G$
is equivalent to 
the category of $\rW$-equivaraint holonomic $D$-module on 
$T$ satisfying condition (3) in Theorem \ref{characterization}.
Since the Verdier duality induces 
an equivalence between the category of 
central $D$-modules and the category of $*$-central $D$-modules,
Theorem \ref{characterization} together with the results in \emph{loc. cit.} imply:
\begin{thm}\label{Whit}
There is an equivalence between  
the category of Whittaker $D$-modules on $G$
and the category of central (resp. $*$-central) $D$-modules on $T$.
\end{thm}

\begin{remark}
In \emph{loc.cit.} Ginzburg and Lonergan also proved that 
the category of Whittaker $D$-modules on $G$ is equivalent to the category of 
holonomic modules over the quantum Toda lattice of 
$G$, and is also equivalent to the category of holonomic modules 
over nil-Hecke algebra of $G$.
\end{remark}

Consider the $\ell$-adic setting.
Fix a non-degenerate homomorphism 
$\chi:U\to\bG_a$, that is, the restriction of 
$\chi$ to each root subgroup $U_\alpha\subset U$ is nontrivial for each 
simple root $\alpha$.
An $\ell$-adic Whittaker sheaf on $G$ is a perverse sheaf 
$\mF$ on $G$ together with an isomorphism
$a^\circ\mF\is\chi^\circ\mL_\psi\boxtimes\chi^\circ\mL_\psi\boxtimes\mF$
satisfying the usual cocycle condition.
Here $a:U\times U
\times G\to G,\ a(u_1,u_2,g)=u_1gu_2^{-1}$.  
We conjecture the following $\ell$-adic counterpart of Theorem \ref{Whit}:

\begin{conjecture}
There is an equivalence between the category of $\ell$-adic Whittaker sheaves on
$G$ and the category of central perverse sheaves on $T$.
\end{conjecture}

\quash{
\begin{remark}
The definition of $*$-central $D$-modules (see Definition \ref{def of central}) makes senses 
without the holonomicity assumption, and Theorem \ref{characterization} in the de Rham setting holds for 
any $\rW$-equivariant $D$-module on $T$. 
Therefore, 
the results in \cite{G2,L} imply that Theorem \ref{Whit} holds without the holonomicity assumption.
\end{remark}
 }

\section{Local systems $\mE_\theta$}\label{tame central local systems}
 
 \subsection{Tame local systems $\mE_\theta$}\label{central Loc}
 In this subsection we attach each $\rW$-orbit $\theta=\rW\chi$ in $\calC(T)(F)$
 a $\rW$-equivariant tame local system $\mE_\theta$ on $T$.

We first consider the de Rham setting.
Define $S$ to be the completion of the group algebra 
$\bC[\pi_1^t(T)]$ with respect to the 
the kernel of the map $\bC[\pi_1^t(T)]\to\bC$ sending $\gamma\in\pi_1^t(T)$ to $1\in\bC$ and let $S_+$ be the argumentation ideal.
The Weyl group $\rW$ acts naturally on 
$S$ and we define 
$S_\chi=S/\langle S^{\rW_\chi}_+\rangle$, where $\langle S^{\rW_\chi}_+\rangle$
is the ideal generated by $S^{\rW_\chi}_+$. Since $\rW_\chi$ is normal in 
$\rW_\chi'$, $S_\chi$ carries an action of $\rW_\chi'$ and we define 
$\rho_\chi^{uni}$ to be the representation of 
$\rW_\chi'\ltimes\pi_1^t(T)$ in the space $S_\chi$ by setting 
$\rho_\chi^{uni}(w,\gamma)v=w(\gamma v)$, 
where $(w,\gamma)\in\rW_\chi'\ltimes\pi_1^t(T)$
and $v\in S_\chi$.
Since $\rW_\chi'$ is the stabilizer of $\chi$ in $\rW$, one can twist $\rho_\chi^{uni}$
by the character $\chi$ and obtain a representation 
$\rho_\chi:=\rho_\chi^{uni}\otimes\chi$
of $\rW_\chi'\ltimes\pi_1^t(T)$ in $S_\chi$ and its induced representation
$\rho_\theta:=\Ind_{\rW_\chi'\ltimes\pi_1^t(T)}^{\rW\ltimes\pi_1^t(T)}\rho_\chi$
of $\rW_{a}^{\on{ex}}=\rW\ltimes\pi_1^t(T)$.
We define
$\mE_\chi$ and $\mE_\theta$ to be the 
$\rW_\chi'$ and $\rW$-equivariant
local systems on $T$ corresponding to $\rho_\chi$ and 
$\rho_\theta$.

We now consider the $\ell$-adic setting.
Choose a finite extension  $K$ of $\barQ$ such that 
elements in $\theta$ are defined over the ring of integer of $O_K$.
Let $S_K=O_K[[\pi_1(T)_\ell]]$
be the completed group algebra of the pro-$\ell$ quotient $\pi_1(T)_\ell$
and let $S_{K,+}$ be the argumentation ideal.
We define $S_{\chi,K}=S_K/\langle S_{K,+}^{\rW_\chi}\rangle$.
Let $\rho_\chi^{uni}$ be the $\ell$-adic representation of
$\rW_\chi'\ltimes\pi_1^t(T)$ in $S_{\chi,K}\otimes_{R_K}\barQ$
given by
$\rho_\chi^{uni}(w,\gamma)(v)=w(\gamma_\ell v)$,
where $\gamma_\ell$ is the image of $\gamma$ under
the quotient  $\pi_1^t(T)\to\pi_1(T)_\ell$. 
Applying the same construction as in the de Rham setting, we obtain 
$\ell$-adic representations 
$\rho_\chi=\rho^{uni}_\chi\otimes\chi$ 
and $\rho_\theta=\Ind_{\rW_\chi'\ltimes\pi_1^t(T)}^{\rW\ltimes\pi_1^t(T)}\rho_\chi$
of 
$\rW_\chi'\ltimes\pi_1^t(T)$
and $\rW\ltimes\pi_1^t(T)$ in
$S_{\chi,K}\otimes_{S_K}\barQ$.
We define $\mE_\chi$ and $\mE_\theta$ to be the 
$\rW_\chi'$ and $\rW$-equivariant
$\ell$-adic local systems on $T$ corresponding to
$\rho_\chi$ and $\rho_\theta$.

\subsection{Mellin transform of $\mE_\theta$}\label{MT of E_xi}
In this subsection we study the Mellin transform of 
$\mE_\theta$.

We first consider the de Rham setting.
For any $\chi\in\calC(T)(\bC)$,
consider the projection map $\pi_\chi:\check\ft\ra\check\ft//\rW_\chi$.
To every $\mu\in\check\ft$, let $m_{\mu}:\check\ft\ra\check\ft, v\ra v+\mu$.
We define 
\[\calR_\chi=\pi_\chi^*\delta,\ \ \ \calR_\chi^\mu=m_{\mu}^*\calR_\chi\]
where $\delta=\delta_0$ the the skyscraper sheaf supported at 
$0\in\check\ft//\rW_\chi$. 
Let $\mu$ be a lift of $\chi$. Then 
we have the following cartiesian diagram 
\[\xymatrix{\check\ft\ar[r]^{m_{\mu}}\ar[d]^{\pi_{-\mu}}&\check\ft\ar[d]^{\pi_\chi}\\
\check\ft//\rW_{a,-\mu}\ar[r]^{\bar m_{\mu}}&\check\ft//\rW_\chi}\]
where $\pi_{\mu}$ is the quotient map and 
$\bar m_{\mu}$ is induced by 
the isomorphism $\rW_{a,-\mu}^{}\is\rW_{\chi^{-1}}=\rW_{\chi}$ in~\eqref{iso of stabilizers}.
It follows that 
$\calR_\chi^\mu\is\pi_{-\mu}^*\delta_{-\bar\mu}$, where 
$\delta_{-\bar\mu}$ is the skyscraper sheaf supported at $-\bar\mu=\pi_{-\mu}(-\mu)$.

The equality $ l_{w(\mu)}\circ w=w\circ l_\mu$
gives rise to an isomorphism 
$\calR_\chi^\mu\is w^*\calR_\chi^{w(\mu)}$ for $w\in\rW_\chi'$.
The isomorphisms 
$\calR_\chi^{\mu+\lambda}\is
l^*_\lambda\calR_\chi^\mu, \lambda\in\Lambda$
and $\calR_\chi^\mu\is w^*\calR_\chi^{w(\mu)},\ w\in\rW_\chi'$, define a $\rW'_\xi\ltimes\Lambda$-equivariant structure on 
\[\calS_\chi=\bigoplus_{\mu\in\lambda+\Lambda}\calR_\chi^\mu.\] 
The induced $\rW\ltimes\Lambda$-equivariant sheaf  
$\Ind_{\rW_\chi'\ltimes\Lambda}^{\rW\ltimes\Lambda}\cS_\chi$ on $\check\ft$ 
depends only on the $\rW$-orbit $\theta\subset\calC(T)(\bC)$ of $\chi$
and we denote it by $\calS_\theta$. 
Note that  
a choice of a lift $\mu$ of $\chi$ gives rise to isomorphisms 
\beq\label{descent of S_chi}
\calS_\chi\is\Ind_{\rW_{a,-\mu}^{\on{ex}}}^{\rW_\chi'\ltimes\Lambda}(\calR_\chi^\mu
)\ \ \ \ \ \ \ \ \ \calS_\theta\is\Ind_{\rW_\chi'\ltimes\Lambda}^{\rW_a^{\on{ex}}}\Ind_{\rW_{a,\mu}^{\on{ex}}}^{\rW_\chi'\ltimes\Lambda}(\calR_\chi^\mu)\is
\Ind_{\rW_{a,\mu}^{\on{ex}}}^{\rW_a^{\on{ex}}}(\calR_{\chi}^\mu).
\eeq

We now
consider the $\ell$-adic setting.
Consider the quotient map
$\pi_\chi:\calC(T)\to\calC(T)//\rW_\chi$.
Let $0\in\calC(T)(\barQ)$ be the trivial character and let 
$\pi_\chi(0)$ be its image in $\calC(T)//\rW_\chi$.
 Let $\mO_{\pi_\chi(0)}$ be the structure sheaf of the point $\pi_\chi(0)$
 and we define 
\beq\label{descent of R_chi}
\calR_\chi:=\pi_\chi^*\mO_{\pi_\chi(0)}
\eeq
which a $\rW_\chi'$-equivariant coherent sheaf on $\calC(T)$.
Define 
$\calS_\chi:=m_{\chi^{-1}}^*(\calR_\chi)$ where
$m_\chi:\calC(T)\to\calC(T)$ be the morphism of translation by $\chi$.
Since $m_\chi$ intertwines the $\rW_\chi'$-action on $\calC(T)$, 
$\calS_\chi$ is $\rW_\chi'$-equivariant,
moreover, there is natural isomorphism 
\beq\label{w action}
w^*\calS_\chi\is\calS_{w^{-1}\cdot\chi}
\eeq for any $w\in\rW$.
For any $\rW$-orbit $\theta$ in $\calC(T)$, we define 
the following $\rW$-equivariant coherent sheaf 
with finite support
\[\mS_\theta=\bigoplus_{\chi\in\theta}\calS_\chi\]
where the $\rW$-equivariant structure is given by the 
isomorphisms in~\eqref{w action}.

Note that 
$\calS_\theta$ is set theoretically supported on $\theta^{-1}=\{\chi^{-1}|\chi\in\theta\}\subset\calC(T)(\barQ)$ in the $\ell$-adic setting, and on
$\{-\mu|\exp(\mu)\in\theta\}\subset\check\ft$ in the de Rham setting.

\begin{lemma}\label{Mellin of E_theta}
There is an isomorphism
\[\mathfrak M(\mE_\theta\otimes\sign)\is\calS_{\theta}\]
compatible with the $\rW$-equivariant structures in the 
 in the $\ell$-adic setting, and the $\rW_a^{\text{ex}}$-equivariant structures in the 
 de Rham setting.

\end{lemma}
\begin{proof}
The $\ell$-adic setting.
Pick a $\chi\in\theta$.
It suffices to show that there is an isomorphism 
$\mathfrak M(\mE_\chi\otimes\sign)\is\calS_\chi$ 
compatible with
the $\rW_\chi'$-equivariant structures.
The $\barQ[[\pi_1(T)_\ell]]$-module  
corresponding to $\mE_\chi\otimes\mL_{\chi^{-1}}$ is isomorphic to $\calR_\chi$.
Since $\calR_\chi\is\on{inv}^*\calR_{\chi}$ (here $\on{inv}$ the inverse map on $\calC(T)$), 
 by \cite[Corollary 4.2.2.4]{GL}, 
there is an isomorphism 
\[\mathfrak M(\mE_\chi)\is m_{\chi^{-1}}^*\mathfrak M(\mE_\chi\otimes\mL_{\chi^{-1}})\is
m_{\chi}^*(\calR_\chi\otimes_{\bZ_\ell}\wedge^{\on{top}}_{\bZ_\ell}(\pi_1(T)_\ell)^\vee)\is
\mS_\chi\otimes_{\bZ_\ell}\wedge^{\on{top}}_{\bZ_\ell}(\pi_1(T)_\ell)^\vee,\]
compatible with the $\rW_\chi'$-actions.
By choosing a generator of $\wedge^{\on{top}}_{\bZ_\ell}(\pi_1(T)_\ell)^\vee$,
we obtain a $\rW$-equivariant isomorphism 
$\wedge^{\on{top}}_{\bZ_\ell}(\pi_1(T)_\ell)^\vee\is\bZ_\ell\otimes\sign$.
The desired claim follows.

The de Rham setting. 
We have an isomorphism of $\mO(\check\ft)$-modules $\Gamma(\mE_\theta)\is\Gamma(\calS_{\theta^{-1}})$
compatible with the 
$\rW_a^{\text{ex}}$-actions.
Thus, by Remark \ref{global section},
there is an isomorphism
$\mathfrak M(\mE_\chi\otimes\sign)\is(-1)^*\calS_{\chi^{-1}}\is\calS_{\chi}$
intertwines the $\rW_a^{\text{ex}}$-equivariant structures.

\end{proof}

\begin{corollary}\label{E_theta is central}
$\mE_\theta$ is $*$-central.
\end{corollary}
\begin{proof}
It is proved in \cite[Corollary 5.2]{C1} in the $\ell$-adic setting.
In the de Rham setting,
Lemma \ref{Mellin of E_theta} 
and the construction of $\calS_\theta$ imply the
Mellin transform 
$\mathscr M(\mE_\theta\otimes\sign)$ satisfies condition (3) in
Theorem \ref{characterization of central}, hence is $*$-central.
\quash{
The case when $\on{char}k>0$ is proved in \cite[Lemma 4.4]{C1}.
Assume $k=\bC$. Choose $\chi\in\theta$ and let $\mu\in\check\ft$ be a lifting of 
$\chi$.
By Lemma \ref{chara of central}, 
it suffices to show that the action of 
$\rW_{\chi^{-1}}\is\rW_{a,-\mu}$ on 
\[\oH^*(T,\mE_\theta\otimes\mL_{\chi^{-1}})\is i_{-\mu}^*\frakM(\mE_\theta)\is i_{-\mu}^*\calS_{\chi}\is\calR_{\chi}^{\mu}\]
is given by the sign character.
This follow from Lemma \ref{Mellin of E_theta} and \ref{descent} and the fact that 
$\calR_{\chi}^{\mu}\is\pi_{-\mu}^*\delta_{-\bar\mu}$ descends to $\check\ft//\rW_{a,-\mu}$.}

\end{proof}

\subsection{Convolution with $\mE_\theta$}

Recall the following descent criterion for coherent complexes \cite{Ne}:
\begin{lemma}\label{descent}
Let $X=\Spec(A)$ be an affine Noetherian scheme over $F$. Let $H$ be a finite group acting on $X$.
Let $\mF\in D^b_{coh}(X/H)$ be a $H$-equivaraint complex of 
coherent sheaves on $X$. The following are equivalent
\begin{enumerate}
\item
For any closed point $x\in X$, the action of 
the stabilizer $G_x$ of $x$ in $G$ on the derived fiber 
$i_x^*\mF$ is trivial. Here $i_x:x\to X$ is the embedding.
\item
$\mF$ descends to $X//H=\Spec(A^H)$.
\end{enumerate}
\end{lemma}
\begin{proof}
By \cite[Theorem 1.1]{MFK}, $X\to X//G$ is a universal geometric quotient, in particular, a 
good quotient in the sense of \cite{Ne}. Now the result follows from 
\cite[Theorem 1.3]{Ne}\footnote{In \emph{loc. cit.} the Theorem is proved for schemes of finite type over $F$, but the same argument works for Noetherian schemes over $F$.}.
\end{proof}

\begin{prop}\label{conv with E_theta sign}
Let $\mF$ be a strongly $*$-central complex and 
let $\theta=\rW\chi$ be a $\rW$-orbit of a tame character $\chi$.
There is an isomorphism 
\[
\mF*(\mE_\theta\otimes\sign)\is\oH^*(T,\mF\otimes\mL_{\chi}^{-1})\otimes\mE_\theta\in\sD_\rW(T)
\]

\end{prop}

\begin{proof}
The $\ell$-adic setting. 
For any tame character $\chi$,
let $D_{\chi}$ (resp. $D'_{\chi}$) be the fiber of 
$\pi_{\chi}:\calC(T)\to\calC(T)//\rW_\chi$ (resp. $\pi_{\chi}':\calC(T)\to\calC(T)//\rW_\chi'$) over 
$\pi_{\chi}(\chi^{-1})$  (resp. $\pi_{\chi}'(\chi^{-1})$).
For any $\rW$-orbit $\theta=\rW\chi$,
let $D_{\theta}=\bigsqcup_{\chi\in\theta} D_{\chi}$
and $D'_{\theta}=\bigsqcup_{\chi\in\theta} D'_{\chi}$.
We have closed embeddings $D_{\chi}\to D_{\chi}'$
and $D_{\theta}\to D_{\theta}'$
, which are isomorphisms if 
$\rW_{\chi}=\rW_{\chi}'$.

By Lemma \ref{Mellin of E_theta}, we have 
\beq
\label{mellin of E}
\frak M(\mE_\theta\otimes\sign)\is\cS_\theta\is\mO_{D_{\theta}}\in\sD_\rW(T).
\eeq
Thus there is an isomorphism 
\beq\label{mellin of conv 1}
\frak M(\mF*\mE_\theta\otimes\sign)\is \frak M(\mF)\otimes\frak M(\mE_\theta\otimes\sign)\is
\frak M(\mF)|_{D_\theta}\is \frak M(\mF)|_{D_{\theta}}.
\eeq
Since $\chi^{-1}$ is the unique closed point of $D'_{\chi}$, by Lemma \ref{chara of central} and Lemma
\ref{descent}, the restriction $\frak M(\mF\otimes\sign)|_{D_{\chi}'}$ descends to 
$D_{\chi}'//\rW_{\chi}'$ and it implies there is an isomorphism
\beq\label{mellin of conv 2}
\frak M(\mF\otimes\sign)|_{D_{\chi}}\is\frak M(\mF)|_{\chi}
\otimes\mO_{D_{\chi}}
\is\oH^*(T,\mF\otimes\mL_{\chi^{-1}})\otimes\mO_{D_{\chi}}
\eeq
compatible with the $\rW_\chi'$-equivariant structure.
It follows that there is an isomorphism 
\beq\label{mellin of conv 3}
\frak M(\mF\otimes\sign)|_{D_{\theta}}
\is\oH^*(T,\mF\otimes\mL_{\chi^{-1}})\otimes\mO_{D_{\theta}}\in\sD_{\rW}(T).
\eeq
All together, we obtain
\[\frak M(\mF*\mE_\theta\otimes\sign)\stackrel{\eqref{mellin of conv 1}}\is (\frak M(\mF)|_{D_{\theta}})\stackrel{\eqref{mellin of conv 3}}\is
\oH^*(T,\mF\otimes\mL_{\chi^{-1}})\otimes\mO_{D_{\theta}}\otimes\sign\stackrel{\eqref{mellin of E}}\is
\oH^*(T,\mF\otimes\mL_{\chi^{-1}})\otimes\frak M(\mE_\theta).\]
Since the Mellin transform restricts to an equivalence on monodronic sheaves~\eqref{inverse Mellin}, the isomorphism above comes from an isomorphism
\beq\label{iso of W}
\mF*(\mE_\theta\otimes\sign)\is\oH^*(T,\mF\otimes\mL_{\chi^{-1}})\otimes\mE_\theta\in\sD_\rW(T).
\eeq
The de Rham setting. 
Choose a lifting $\lambda\in\check\ft$ of $\chi$ and 
write $D_{\lambda}$
and $D'_{\lambda}$ for the fiber of 
$\pi_{-\lambda}:\check\ft\to\check\ft//\rW_{a,-\lambda}$ and $\pi_{\lambda}':\check\ft\to\check\ft//\rW_{a,-\lambda}^{\on{ex}}$ over 
$\pi_{-\lambda}(-\lambda)$ and $\pi_{-\lambda}'(-\lambda)$,
and let $D_{\theta}=\bigsqcup_{\lambda,\exp\lambda\in\theta} D_\lambda$ and $D_{\theta'}=\bigsqcup_{\lambda,\exp\lambda\in\theta} D'_\lambda$.
Now the  
same argument as in the $\ell$-adic setting, replacing 
$D_\chi$, $D'_{\chi}$, $\rW_\chi$, and $\rW_\chi'$ by
$D_\lambda$, $D'_{\lambda}$, $\rW_{a,-\lambda}$, and $\rW_{a,-\lambda}^{\on{ex}}$, gives the desired isomorphism~\eqref{iso of W}.

\end{proof}

\begin{remark}
Note that
if we only assume $\mF$ is $*$-central, then the isomorphism in~\eqref{mellin of conv 2}
is only compatible with the $\rW_\chi$-equivariant structures.
As a result, the isomorphisms~\eqref{iso of W} 
is not compatible with the $\rW$-equivariant structure.
\end{remark}

The proposition above can be reformulated as follows:
\begin{prop}\label{conv with E_theta}
Let $\mF$ be a strongly $*$-central complex and 
let $\theta=\rW\chi$ be a $\rW$-orbit of a tame character $\chi$.
For each $w\in\rW$, there is a canonical isomorphism in $\sD(T)$
\beq\label{a_w}
a_w:w^*\mF*\mE_\theta\is\oH^*(T,\mF\otimes\mL_{\chi}^{-1})\otimes w^*\mE_\theta,
\eeq
such that the following diagram is commutative
\beq\label{compatibility}
\xymatrix{w^*\mF*\mE_\theta\ar[r]^{a_w\ \ \ \ \ \ \ \ \ }\ar[d]^{}&\oH^*(T,\mF\otimes\mL_{\chi^{-1}})\otimes w^*\mE_\theta\ar[d]\\
\mF*\mE_\theta\ar[r]^{a_e\ \ \ \ \ \ \ \ \ }&\oH^*(T,\mF\otimes\mL_{\chi^{-1}})\otimes\mE_\theta},
\eeq
where the vertical arrows are the isomorphism induced from the 
$\rW$-equivaraint structures on $\mF$ and $\mE_\theta$.

\end{prop}

\section{Categorical center and the character sheaves $\calM_\theta$}

Let $\mE_\theta$ be the $*$-central local system in Section~\ref{tame central local systems}. 
We define $\cM_\theta:=\Phi_{\mE_\theta}=\Ind_{T\subset B}^G(\mE_\theta)^\rW$.
Recall the averaging functor 
$\Av_*^U:=\pi_*:\sD(G)\to\sD(G/U)$, where $\pi:G\to G/U$ is the quotient map.
The goal of this section is to poof the following theorem:

\begin{thm}\label{Key}
There exists a positive integer $N$ depending only on the type of the group $G$ such that 
the following holds. 
Assume $k=\bC$ or $\on{char}k=p$ is not dividing
$N\ell$.
We have 
\[\on{Av}_*^U(\cM_{\theta})\is\mE_\theta.\]
In particular, $\on{Av}_*^U(\cM_{\theta})$
is supported on $T=B/U\subset G/U$.
\end{thm}

We will fist establish Theorem \ref{Key} in the case $k=\bC$ using the 
results in \cite{BFO}, \cite{BG}, and \cite{MS} on 
Harich-Chandra bimodules, categorical centers of Hecke categories, and
Character $D$-modules.  Then we construct a mixed characteristic lifting
$\cM_{\theta,A}$
of $\cM_\theta$ over a strictly Henselian ring $A$ with residue field $k$
of characteristic $p$ not dividing $N\ell$, which is
 universal local acyclic with respect to the quotient map $G_A\to G_A/U_A$.
This allows us to
use spreading out arguments to prove Theorem \ref{Key} in positive characteristic case.

We will assume $k=\bC$ until Section \ref{mixed char}.

\subsection{Hecke categories}
Consider the left $G$ and right $T\times T$ actions on $Y=G/U\times G/U$.
To every $\chi,\chi'\in\check T$ 
we denote by $M_{\chi,\chi'}$ the category of $G$-equivariant 
$D$-modules on
$G/U\times G/U$ which are $T\times T$-monodromic 
with 
generalized monodromy $(\chi,\chi')$, that is, 
$U(\ft)\otimes U(\ft)$ (acting as infinitesimal
translations along the right action of $T\times T$) acts 
locally finite with generalized eigenvalues in $(\chi,\chi')$.
Consider the quotient  $Y/T$ where $T$ acts diagonally from the right.
The group $T$ acts on $Y/T$ via the formula $t(xU,yU)\on{mod }T=
(xU,ytU)\on{mod} T$. To every $\chi\in\check T$ we denote by
$M_\chi$ the category of $G$-equivariant $T$-monodromic $D$-modules on 
$Y/T$ with generalized monodromy $\chi$.
We write $\sD(M_{\chi,\chi'})$ and $\sD(M_\chi)$ for the corresponding 
$G$-equivariant monodromic derived categories.

The groups $B$ and $T\times T$ act on $X=G/U$ by the 
formula $b(xU)=bxb^{-1}U$, $(t,t')(xU)=txt'U$.
For any $(\chi_1,\chi_2)\in\check T\times\check T$ we write 
$H_{\chi_1,\chi_2}$ for the category of $U$-equivariant $T\times T$-monodromic 
$D$-modules on $X$ with generalized monodromy $(\chi_1,\chi_2)$.
For any $\chi\in\check T$ we write 
$H_{\chi}$ 
for the category of $B$-equivariant $T$-monodromic
$D$-modules on $X$ 
with generalized monodromy $\chi$, where $B$ acts 
on $X$ by the same formula as before and  
$T$ acts on $X$ by the formula $t(xU)=txU$.
We denote by $\sD(H_\chi)$ (resp.
$\sD(H_{\chi_1,\chi_2})$)  the corresponding $B$-equivaraint (resp. $U$-equivaraint)
monodromic derived category.

\subsection{The Harish-Chandra functor}
Consider the following correspondence \[G\stackrel{p}\la G\times G/B\stackrel{q}\ra Y/T=(G/U\times G/U)/T\]
where $p(g,xB)=g$ and $q(g,xB)=(gxU,xU)\on{mod} T$. 
The group $G$ acts on $G$, $G\times \mB$ and $Y/T$ by the formulas 
$a\cdot g=aga^{-1}$, $a\cdot(g,xU)=(aga^{-1},axB)$, $a(xU,yU)=(axU,ayU)$.
One can check that $p$ and $q$ are compatible with those $G$-actions.

Following \cite{G1,MV}, we consider the functor 
\beq\label{HC}
\mathrm{HC}=q_*p^\circ:
\sD(G)\ra
\sD(Y/T).
\eeq The functor above admits a right adjoint 
$\mathrm{CH}=p_*q^\circ:\sD(Y/T)\ra\sD(G)$. 
We use the same notations for the corresponding functors between
$G$-equivariant derived categories
$\sD(G/_\ad G)$ and $\sD(G\backslash Y/T)$.

\quash{
Recall the following well-known fact:
\begin{lemma}[Theorem 3.6 \cite{MV}]\label{CHHC}
\begin{enumerate}
\item
Let $\sigma:\widetilde\cN\ra\cN$ be the Springer resolution of the nilpotent cone $\cN$
and let $Sp:=\sigma_!F_{\widetilde\cN}$ be the Springer $D$-module. 
For any $\mF\in D(G)$
there is canonical isomorphism 
\[\mathrm{CH}\circ\mathrm{HC}(\mF)\is\mF*Sp.\]

\item
We have a canonical isomorphism 
$\mathrm{CH}\circ\mathrm{HC}\is\on{Av}^G_*\circ\on{Av}^U_*[-]$.

\item
The identity functor is a direct summand of 
$\mathrm{CH}\circ\mathrm{HC}\is\on{Av}^G_*\circ\on{Av}^U_*[-]$.

\end{enumerate}

\end{lemma}
}

Consider the embedding $i:X\ra Y, gU\ra (eU,gU)$ and the projection map
$\pi:G\to X=G/U$.
\begin{lemma}\cite{MV}\label{M and H}
(1) The functor $i^0=i^![\dim X]:\sD(G\backslash Y)\ra \sD(U\backslash X)$ is an equivalence of categories with inverse 
givne by $(i^0)^{-1}:=\on{Ind}_{T\subset B}^G\circ i_*[\on{dim}G-\on{dim}B]$.
(2) We have $i^0\circ\on{HC}\is\pi_*$.

\end{lemma}

We have the convolution product
$\sD(G\backslash Y)\times \sD(G\backslash Y)\ra \sD(G\backslash Y)$ given by $(\mF,\mF')\ra (p_{13})_*(p_{12}^*\mF\otimes p_{23}^*\mF')$.
Here $p_{ij}:G\backslash (G/U\times  G/U\times  G/U)\ra G\backslash Y=G\backslash (G/U\times G/U)$
is the projection on the $(i,j)$-factors. 
The convolution product on $\sD(G\backslash Y)$ restricts to a convolution product on 
$\sD(M_{\chi,\chi^{-1}})$. 
The equivalence $i^0:\sD(G\backslash Y)\is\sD(U\backslash X)$
above induces convolution products on $\sD(U\backslash X)$ and $\sD(H_{\chi,\chi})$. 
In addition, there is an action of $\sD(U\backslash X)$ on $\sD(X)$ by right convolution. 
The convolution operation will be denoted by $*$.

\quash{
\begin{lemma}[Proposition 9.2.1 \cite{G}]\label{central}
For any $\cM\in D_G(G)$, $\cF\in D_G(Y)$, and $\cF'\in D_U(X)$ we have 
\[\on{HC}(\cM)*\cF\is\cF*\on{HC}(\cM),\ \  \on{Av}_U(\cM)*\cF'\is\cF'*\on{Av}_U(\cM).\]
\end{lemma}
}

\quash{
\subsection{(Pro) local system $\hat\mL_\xi$}
Let $\hat\mL$ be the \emph{pro-unitpotnet} local system on $T$, that is, 
$\hat\mL=\underleftarrow{\on{lim}}\ \mL_n$ where $\mL_n$ is the local system whose fiber at the unit
element $e\in T$ is identified with 
$\on{Sym}(\ft)/\on{Sym}(\ft)_+^n$
, where the log monodromy action of $\pi_1(T)$
coincides with the 
restriction of 
natural $\on{Sym}(\ft)$-module structure to $\pi_1(T)\subset
\pi_1(T)\otimes\bC\is
\ft$.
Let $\xi\in\breve T$ and $\mL_\xi$ be the corresponding Kummer local system on $T$. We define 
the following pro-local system
\[\hat\mL_\xi=\underleftarrow{\on{lim}}(\mL_n\otimes\mL_\xi).\]

}

We will need the following lemma. 
Let $X$ be an algebraic variety with an action 
of an affine algebraic group $G$. Denote the 
action map by $a:G\times X\ra X$ .
\begin{lemma}[Lemma 2.1 \cite{BFO}]\label{action}
For any 
$\cA\in\sD(G)$, $\cF\in\sD(X)$ 
We have a canonical isomorphism 
\[R\Gamma(a_*(\mA\boxtimes\mF))\is R\Gamma(\mA)\otimes^L_{U(\fg)}R\Gamma(\cF).\]
\end{lemma}

\begin{example}\label{formula for Av}
Consider the action map 
$a:G\times G/U\to G/U, a(x,gU)=xgU$.
Let $\delta\in\sD(G/U)$ be the delta $D$-module supported at the base point $eU\in G/U$.
For any $D$-module $\mF$ on $G$, there is a canonical isomorphism 
$\Av_*^U(\mF)\is a_*(\mF\boxtimes\delta)$ and lemma above implies that 
\[
R\Gamma(\Av_*^U(\mF))\is R\Gamma(a_*(\mF\boxtimes\delta))\is R\Gamma(\mF)\otimes^L_{U(\fg)}R\Gamma(\delta)
.\]
Note that $R\Gamma(\mF)=\Gamma(\mF)$ (since $G$ is affine) and $R\Gamma(\delta)\is U(\fg)/U(\fg)\fn$, and it follows that 
\[
R\Gamma(\Av_*^U(\mF))\is\Gamma(\mF)\otimes_{U(\fg)}^LU(\fg)/U(\fg)\fn
\]
\end{example}
\subsection{Character $D$-modules} 
We denote by $CS(G)$ the 
category of finitely generated $G$-equivaraint $D$-modules on $G$ such that 
the action of the center $Z\subset U(\fg)$, embedding as left invariant 
differential operators, is locally finite.
To every $\rW$-orbit $\theta\subset\check T$,
we denote by $CS_\theta(G)$ the 
category of finitely generated $G$-equivaraint $D$-modules on $G$ such that 
the action of the center $Z\subset U(\fg)$ is locally finite and has generalized eigenvalues in $\{\lambda\in\check\ft|\exp(\lambda)\in\theta\}$.
We denote by 
$\sD(CS(G))$ (resp. $\sD(CS_\theta)$)
 the minimal triangulated full subcategory of 
$\sD(G/_\ad G)$ containing all objects $\mM\in\sD(G/_\ad G)$ such that $\sH^i(\mM)\in CS(G) $
(resp. $\sH^i(\mM)\in CS_\theta(G)$ ).
We call 
$CS(G)$ and $CS(G)_\theta$ (resp. $\sD(CS(G))$ and $\sD(CS_\theta)$) 
the category (resp. derived category)
of 
character $D$-modules on $G$ and 
character $D$-modules on $G$ 
with generalized central character $\theta$.

\begin{proposition}\label{CS}
We have the following:
\begin{enumerate}
\item 
Let $\mG\in CS(G)_\theta$. Then 
\[\on{HC}(\mG)\in\bigoplus_{\chi\in\theta}\sD(M_{\chi}),\ \ (resp.\ \on{Av}_*^U(\mG)\in
\bigoplus_{\chi\in\theta}\sD(H_\chi).
)\]

\item 
The functors $\Ind_{T\subset B}^G$ and $\Res_{T\subset B}^G$ 
preserve the derived categories of character $D$-modules. 
The resulting functors $\Ind_{T\subset B}^G:\sD(CS(T))\ra \sD(CS(G))$,
$\Res_{T\subset B}^G:\sD(CS(G))\ra \sD(CS(T))$  
are independent of the choice of the Borel subgroup $B$ and 
t-exact with respect to the natural $t$-structures on $\sD(CS(G))$
and $\sD(CS(T))$. 
Moreover, for any $\mG\in CS(G)$ we have 
$\on{Res}_{T\subset B}^G(\mG)\is (j_{T})_{!*}(\mG|_{T^{\rs}})$, here
$j_T:T^{\rs}\to T$ is the embedding.

\item
Let $\mG\in CS(G)$. 
There is a canonical $\rW$-equivariant structure on $\on{Res}_{T\subset B}^G(\mG)$.
Let $j:G^{\rs}\to G$ be the open embedding.
If $\mG=j_{!*}j^*\mG$, then we have 
\[\mG\is\Ind_{T\subset B}^G(\on{Res}_{T\subset B}^G(\mG))^{\rW}.\]
\quash{
\item
Let $CS_{[T]}(G)\subset CS(G)$ be the full subcategory generated by 
the image of $\Ind_{T\subset B}^G:CS(T)\ra CS(G)$.
For any $\mG\in CS_{[T]}(G)$, the local system
$\mF=\Res_{T\subset B}^G(\mG)\in CS(T)$ carries a canonical 
$\rW$-equivariant structure, moreover, there is a canonical isomorphism 
\[
\Ind_{T\subset B}^G(\mF)^{\rW}\is\mG.
\]
Here $\Ind_{T\subset B}^G(\mF)^{\rW}$ is the $\rW$-invaraint factor of
$\Ind_{T\subset B}^G(\mF)$
for the $\rW$-action constructed in Proposition \ref{properties of ind}.
}

\end{enumerate}

\end{proposition}
\begin{proof}
Part (1) and (2) are proved in \cite{Gi1,Lu}.
We now prove part (3). We first show that $\mF=\Res^G_{T\subset B}(\mG)$ is 
canonically $\rW$-equivariant.
Let $x\in N(T)$ and $w\in N(T)/T=\rW$ be its image in the Weyl group.
Denote $B_x:=\Ad_xB$.
Consider the following commutative diagram 
\[\xymatrix{T\ar[d]^{w}&B\ar[l]\ar[r]\ar[d]^{\Ad_x}&G\ar[d]^{\Ad_x}\\
T&B_x\ar[l]\ar[r]&G}
\]
where 
$w:T\ra T$ the natural action of $w\in\rW$ on $T$ and 
the horizontal arrows are the natural inclusion and projection maps.
The base change theorems and the fact that the functors 
$\Res^G_{T\subset B}$ and $\Res^G_{T\subset B_x}$ are canonical isomorphic (see part (2)) imply 
\beq\label{base change}
\Res^G_{T\subset B}(\Ad_x^*\mG)\is w^*\Res^G_{T\subset B_x}(\mG)\is w^*
\Res^G_{T\subset B}(\mG).
\eeq
Since $\mG$ is $G$-conjugation equivariant, we have a 
canonical isomorphism 
$c_x:\mG\is\Ad^*_x\mG$.
Applying $\Res^G_{T\subset B}$ to $c_x$  
and using  (\ref{base change}) we get 
\beq\label{c_w'}
\mF=
\Res^G_{T\subset B}(\mG)\is\Res^G_{T\subset B}(\Ad_x^*\mG)\is w^*\Res^G_{T\subset B}(\mG)
=w^*\mF.
\eeq
We claim that the isomorphism above depends only the image $w$
and we denote it by 
\beq\label{c_w}
c_w:\mF\is w^*\mF.
\eeq
To prove the claim it is enough to check that for $x\in T$ the 
restriction of the isomorphism (\ref{c_w'}) to $T^{\rs}$ 
is equal to the identity map. By \cite{Gi1}, the restriction 
$\mF|_{T^{\rs}}$ is canonically isomorphic 
to $\mG|_{T^{\rs}}$ and the map in (\ref{c_w'}) is equal to the 
restriction of $c_x$ to $T^{\rs}$. 
Since 
the adjoint action $\Ad_x:G\ra G$ is trivial on $T$, 
the claim follows from 
the fact that 
any $T$-equivariant structure of a local system on $T$
is trivial.
The $G$-conjugation equivariant structure on $\mG$ implies 
$\{c_w\}_{w\in\rW}$ satisfies the required cocycle condition, hence,
the data $(\mF,\{c_w\}_{w\in\rW})$ defines a $\rW$-equivariant structure on 
$\mF=\Res^G_{T\subset B}(\mG)$.
We shall prove
$\Ind_{T\subset B}^G(\cF)^\rW\is\cG$.
Let $c_{\rs}:G^{\rs}\ra T^{\rs}//\rW$ 
the restriction of the Chevalley map $c:G\ra T//\rW$ to $G^{\rs}$.
Since
$\mG\is j_{!*}(\mG|_{G^{\rs}})$ and $\Ind^G_{T\subset B}(\mF)^\rW\is
j_{!*}(c_{\rs}^*(\tilde\mF))$, where 
$\tilde\mF\in\sD(T^{\rs}//\rW)$ is the descent of 
$\mF|_{T^{\rs}}$ along the map $q_{\rs}:T^{\rs}\ra T^{\rs}//\rW$, it suffices to show $\mG|_{G^{\rs}}\is c_{\rs}^*(\tilde\mF)\in\sD(G^{\rs}/G)$.
This follows again from the fact that 
$\mG|_{T^{\rs}}\is\mF|_{T^{\rs}}\in\sD(T^{\rs}/\rW)\is\sD(G^{\rs}/G)$.

\end{proof}

\subsection{Drinfeld center of Harish-Chandra bimodules
and character $D$-modules}\label{center}
We give a review of the work
\cite{BG,BFO} on  Drinfeld center of Harish-Chandra bimodules
and character $D$-modules.
 
Write $\on{U}=U(\fg)$ for the universal enveloping algebra of $\fg$. 
Let $Z=Z(\U)$ be the center of $\U$. 
Consider the dot action $w\cdot\lambda=w(\lambda+\rho)-\rho$ on $\check\ft$.
Denote by $\dot{\rW}$ the group $\rW$ acting via the dot action $\check\ft$.
We have the Harish-Chandra isomorphism 
$hc:Z\is \mO(\check\ft)^{\dot{\rW}}$ such that for any $\lambda\in\check\ft$
the center $Z$ acts on the Verma module associated to $\lambda$ via 
$z\ra hc(z)(\lambda)$.
For any $\lambda\in\check\ft$  we write 
$m_\lambda$ for the corresponding maximal ideal 
and denote by $I_\lambda$ the maximal ideal  of $Z$ corresponding to 
$m_\lambda$ under the Harish-Chandra isomorphism.
Consider the extended universal enveloping algebra $\tU=\U\otimes_{Z}\mO(\check\ft)$, where $Z$ acts on $\mO(\check\ft)$ via the Harish-Chandra isomorphism.
We denote by $\U_\lambda=\U/\U I_\lambda$, $U_{\hat\lambda}=\underleftarrow{\on{lim}}(\U/\U I_\lambda^n)$,
$\tU_\lambda=\tU/\tU m_\lambda$, and $\tU_{\hat\lambda}=\underleftarrow{\on{lim}} (\tU/\tU m_\lambda^n)$.
The action of $\dot{\rW}$ on $\mO(\check\ft)$ gives rise to an action of 
$\dot{\rW}$ on $\tU$ such that $\tU^{\dot{\rW}}=\U$. In addition,
the stabilizer $\dot{\rW}_\lambda$ of $\lambda\in\check\ft$ in $\dot{\rW}$ acts naturally on 
$\tU_{\hat\lambda}$ and 
the natural inclusion $\U\to\tU$ induces an isomorphism 
$\U_{\hat\lambda}\is\tU_{\hat\lambda}^{\dot{\rW}_{\lambda}}$ (see, e.g., \cite[Section 1]{BG}).

We denote by $\mathcal{HC}_{\hat\lambda}$ the category of 
finitely generated Harish-Chandra bimodules over 
$\U_{\hat\lambda}$, that is, finitely generated continuous $\U_{\hat\lambda}$-bimodules such that the 
diagonal action of $\fg$ is locally finite.
We denote by $\sD(\mathcal{HC}_{\hat\lambda})$ the corresponding derived category.
The tensor product 
$\mM\otimes_
{\U}\mM'$, 
$\mM,\mM'\in\mathcal{HC}_{\hat\lambda}$
(resp. $\mM\otimes^L_{\U}\mM'$, $\mM,\mM'\in\sD(\mathcal{HC}_{\hat\lambda})$) defines a monoidal structure on $\mathcal{HC}_{\hat\lambda}$
(resp. $\sD(\mathcal{HC}_{\hat\lambda})$).

Recall that 
$\lambda\in\check\ft$ is called regular 
if $\dot{\rW}_{\lambda}=0$, that is, $\lambda$ does not lie on any coroot
hyperplane shifted by $-\rho$, and it is called dominant if the value of 
$\lambda$ at any positive coroot is not a negative integer.

\begin{proposition}\cite[Proposition 3.1]{BFO}\label{monoidal}
Let $\chi\in\check T$ and $\lambda\in\check\ft$ be a dominate regular lifting of $\chi$. 
The functor 
\[R\Gamma^{\hat\lambda,\widehat{-\lambda-2\rho}}:(\sD(M_{\chi,\chi^{-1}}),*)\is (\sD(\mathcal{HC}_{\hat\lambda}),\otimes^L_{\U})\] is an 
equivalence of monoidal categories.

\end{proposition}

Let $\chi\in\check T$ and $\lambda\in\check\ft$ be as in Proposition \ref{monoidal} 
and consider the
equivalence of monoidal categories
\beq\label{equ M}
\mathbf M:(\sD(H_{\chi,\chi}),*)\is (\sD(M_{\chi,\chi^{-1}}),*)\stackrel{R\Gamma^{\hat\lambda,\widehat{-\lambda-2\rho}}}\is (\sD(\mathcal{HC}_{\hat\lambda}),\otimes^L_{\U}).
\eeq

We denote by 
$\mathcal{HC}$ the category of 
finitely generated Harish-Chandra bimodules over 
$\U$ (no restriction on the action of the center $Z$).
We denote by
$\mathcal{HC}_{\hat\chi}$ the category of 
finitely generated Harish-Chandra bimodules over
the product $\prod_{\mu\in\lambda+\Lambda/\rW_\chi'}\U_{\hat\mu}$ (here $\rW_\chi'$ acts on
$\mu\in\lambda+\Lambda$ via the dot action).
We denote by  
$Z(\mathcal{HC}_{\hat\lambda},\otimes_{\U})$ (resp. $Z(\mathcal{HC},\otimes_{\U})$, 
$Z(\mathcal{HC}_{\hat\chi},\otimes_{\U})$, $Z(H^t_{\xi,\xi},*^t)$) the Drinfeld center of the monoidal category $(\mathcal{HC}_{\hat\lambda},\otimes_{\U})$ (resp. $(\mathcal{HC},\otimes_{\U})$, $(\mathcal{HC}_{\hat\chi},\otimes_{\U})$, 
$(H^t_{\chi,\chi},*^t)$). Recall an element in $Z(\mathcal{HC}_{\hat\lambda},\otimes_{\U})$ consists of an element 
$\mM\in \mathcal{HC}_{\hat\lambda}$ together with family of compatible isomorphisms $b_\mF:\mM\otimes_{\U}\mF\is\mM\otimes_{\U}\mF$
for $\mF\in \mathcal{HC}_{\hat\lambda}$. 
Recall the notion of translation functor $\theta_\lambda^\mu:\mathcal{HC}_{\hat\lambda}\ra\mathcal{HC}_{\hat\mu}$
where $\mu\in\lambda+\Lambda$ (see, e.g., \cite{BG}).

\begin{thm}\label{BFO}
(1) \cite[Lemma 3.7]{BFO} There is a lifting  $\theta_\lambda^\mu:Z(\mathcal{HC}_{\hat\lambda},\otimes_{\U})\ra Z(\mathcal{HC}_{\hat\mu},\otimes_{\U})$ such that 
the functor
$\mathbf F:Z(\mathcal{HC}_{\hat\lambda},\otimes_{\U})\ra Z(\mathcal{HC}_{\hat\chi},\otimes_{\U}),\ 
L\ra\bigoplus_{\mu\in\lambda+\Lambda/\rW'_\chi}\theta_\lambda^\mu(L)$
define an equivalence of braided monoidal categories.

(2) \cite[Theorem 3.6, Lemma 3.8]{BFO} For any  $\mM\in\sD_G(G)^\heartsuit$, the global section  $\Gamma(\mM)$
is naturally a 
Harish-Chandra bimodule, with a canonical central structure and 
the resulting functor  $\Gamma:\sD_G(G)^\heartsuit\ra Z(\mathcal{HC},\otimes_{\U})$
is 
an equivalence of abelian categories. 
Moreover, the equivalence above 
restricts to an equivalance $CS_\theta\is Z(\mathcal{HC}_{\hat\chi},\otimes_{\U})$
and the composed equivalence 
\[CS_\theta\is Z(\mathcal{HC}_{\hat\chi},\otimes_{\U})
\stackrel{\mathbf F^{-1}}\is Z(\mathcal{HC}_{\hat\lambda},\otimes_{\U})\]
is isomorphic to $R\Gamma^{\hat\lambda,\widehat{-\lambda-2\rho}}\circ\pi^\circ\circ\on{HC}$.
Here $\pi:Y\ra Y/T$ is the projection map.
\end{thm}

\subsection{Harish-Chandra bimodules $\cZ_\lambda$}
In this subsection we attach to each dominant regular lift $\lambda\in\check\ft$ of $\chi\in\check T$
an element 
$\cZ_\lambda\in Z(\mathcal{HC}_{\hat\lambda},\otimes_{\U})$  
in the Drinfeld center of Harish-Chandra bimodules, and identify it with 
the local system $\mE_\chi$ under the equivalence in~\eqref{equ M}.

Let $R=\mO(\check\ft)$ and 
let
$\lambda\in\check\ft$ be a lift of $\chi$.
We have a canonical identification 
$\rW_{\lambda+\rho}\is\dot{\rW}_\lambda, w\to \dot{w}=w$ such that 
the translation map
$m_{-\lambda}:\check\ft\to\check\ft,\  m_{-\lambda}(\mu)=\mu-\lambda$
satisfies $m_{-\lambda}(wv)=\dot{w}(m_{-\lambda}(v))$
for $w\in\rW_{\lambda+\rho}$ and $v\in\check\ft$.
It
indues an isomorphism 
$m_{-\lambda}^*:R_{\hat{0}}^{\rW_{\lambda+\rho}}\is R_{\hat{\lambda}}^{\dot{\rW}_\lambda},\ f\ra (x\ra f(x-\lambda))$. 
Here $R_{\hat{0}}^{\rW_{\lambda+\rho}}$ (resp. $R_{\hat{\lambda}}^{\dot{\rW}_\lambda}$)
is the $\rW_{\lambda+\rho}$-invariant of the completion $R_{\hat 0}$ of $R$ at $0$ (resp.
$\dot{\rW}_{\lambda}$-invariant of the completion $R_{\hat{\lambda}}$ of $R$ at $\lambda$
).
Consider the following  
quotient
$R_\chi^{\lambda+\rho}=R_{\hat 0}^{\rW_{\lambda+\rho}}/\langle R^{\rW_\chi}_{\hat 0,+}\rangle$
(here
$\langle R^{\rW_\chi}_{\hat0,+}\rangle$
is the ideal generated by the augmentation ideal $R^{\rW_\chi}_{\hat 0,+}$ of $R^{\rW_\chi}_{\hat 0}$).
We introduce the following Harish-Chandra bimodule
\beq\label{central element}
\cZ_\lambda=\U_{\hat\lambda}\otimes_{R_{\hat0}^{{\rW}_{\lambda+\rho}}}R_\chi^{\lambda+\rho}\in\HC_{\hat\lambda}.
\eeq
Here $R_{\hat0}^{\rW_{\lambda+\rho}}$ acts on $\U_{\hat\lambda}$ via the 
the map 
\beq\label{map b}
b_\lambda:R_{\hat 0}^{\rW_{\lambda+\rho}}\stackrel{m^*_{-\lambda}}\is R_{\hat\lambda}^{\dot{\rW}_\lambda}\ra Z(\tU_{\hat\lambda}^{\dot{\rW}_\lambda})\to
Z(\U_{\hat\lambda}),
\eeq
where the last map is induced from the isomorphism 
$\tU_{\hat\lambda}^{\dot{\rW}_\lambda}\is\U_{\hat\lambda}$.
Equivalently, consider the $R_{\hat\lambda}^{\dot{\rW}_\lambda}$-module
$\dot{R}^{\lambda+\rho}_\chi$ given by the base change of the 
$R_{\hat0}^{\rW_{\lambda+\rho}}$-module
$R^{\lambda+\rho}_\chi$
along the isomorphism $R_{\hat\lambda}^{\dot{\rW}_\lambda}\stackrel{m_{-\lambda}^*}\is R_{\hat0}^{\rW_{\lambda+\rho}}$.
Then we have 
\[\calZ_{\hat\lambda}\is\U_{}\otimes_{Z}\dot{R}_\chi^{\lambda+\rho}\]
where $Z$ acts on $\dot{R}_\chi^{\lambda+\rho}$
via the Harish-Chandra isomorphism $hc:Z\is R^{\dot\rW}$.
\quash{
let $R_{\hat\lambda}$ be the completion of $R$ at $\lambda$
and let $m_{\hat\lambda}$ be its maximal ideal. 
We have $R_\chi^{\lambda}\is R_{\hat\lambda}^{\dot{\rW}_\lambda}/
\langle R_{\hat\lambda,+}^{\rW_{a,\lambda}}\rangle$
where $\rW_{a,\lambda}$ is the stabilizer of $\lambda$ in 
$\rW_a$ and $R_{\hat\lambda,+}^{\rW_{a,\lambda}}=R_{\hat\lambda}^{\rW_{a,\lambda}}\cap m_{\hat\lambda}$.
We introduce
\beq\label{central element}
\calZ_\lambda:=\U\otimes_ZR^{\lambda}_{\chi}\in\HC_{\hat\lambda}
\eeq
where 
$Z$ acts on $R_\chi^{\rW_\lambda}$ via the Harish-Chandra isomorphism.
}

\quash{
We have a natural forgetful map 
$For:Z_\xi\ra M_{\xi,\xi}^t$ sending $(\mM,b_\mF)_{\mF\in H_{\xi,\xi}^t}\ra\mM$.}

\begin{prop}\label{E=Z}
Let 
$\lambda\in\check\ft$ be a dominant regular lift of
$\chi\in\check T$.
\begin{enumerate}
\item To every $\mM\in \mathcal{HC}_{\hat\lambda}$ there is a canonical isomorphism 
\[b_\mM:\mathcal Z_\lambda\otimes_{\U}\mM\is\mM\otimes_{\U}\cZ_\lambda\] such that 
the data $(\cZ_\lambda,b_\mM)_{\mM\in \mathcal{HC}_{\hat\lambda}}$ defines 
an element in the Drinfeld center $Z(\mathcal{HC}_{\hat\lambda},\otimes)$.

\item  
We have $\mathbf M(\cE_\chi)\is\cZ_\lambda$.
\end{enumerate}
\end{prop}
\begin{proof}
Proof of (1).
We have $\rW_{\lambda+\rho}=\dot{\rW}_\lambda=e$ since $\lambda$ is regular.
To every $\mM\in \mathcal{HC}_{\hat\lambda}$, the map $b_\lambda$ in~\eqref{map b} 
gives rise to an action of $R_{\hat 0}\otimes_\bC R_{\hat 0}=R_{\hat0}^{\rW_{\lambda+\rho}}\otimes_\bC R^{\rW_{\lambda+\rho}}_{\hat0}$ on $\mM$ and it follows from \cite[Theorem 4.1]{MS} that
this action factors through $R_{\hat 0}\otimes_\bC R_{\hat0}\ra R_{\hat0}\otimes_{R_{\hat0}^{\rW_\chi}}R_{\hat0}$.
Indeed, it is shown in \emph{loc. cit.} that for every finite dimensional representation $V$ of 
$\fg$, the action of $R_{\hat0}\otimes_\bC R_{\hat0}$ on
the bimodule $\on{pr}_{\hat\lambda}(V\otimes_\bC\U/\U I_\lambda^n)\in\HC_{\hat\lambda},\ n\in\bZ_{\geq0}$ factors through 
$R_{\hat0}\otimes_{R_{\hat0}^{\rW_\chi}}R_{\hat0}$.
Here $\on{pr}_{\hat\lambda}(-)$ is the projection of the summand on which 
the action of $I_\lambda$ is locally nilpotent\footnote{In fact, it was shown  \cite[Theorem 4.1]{MS} that,
for any $\lambda\in\check\ft(\bC)$ and any finite dimensional representation $V$,
the action of $R_{\hat0}\otimes_{}R_{\hat0}$ on $\pr_{\hat\lambda}(V\otimes\U/\U I_\lambda^n)$
factors through 
$R_{\hat0}\otimes_{R_{\hat 0}^\rW} R_{\hat0}$. But the proof in \emph{loc.cit.} actually shows that,
if we fix $\lambda\in\check\ft(\bC)$, 
the action factors through $R_{\hat0}\otimes_{R_{\hat 0}^{\rW_\chi}}R_{\hat0}$.}.
Since every object in $\HC_{\hat\lambda}$ is isomorphic to a quotient of 
$\on{pr}_{\hat\lambda}(V\otimes_\bC\U/\U I_\lambda^n)$ for some $V$ and $n$, the claim follows.
Therefore, for every $\mM\in \mathcal{HC}_{\hat\lambda}$, we have a 
canonical
isomorphism \[b_\mM:\cZ_\lambda\otimes_{\U}\mM
\is R_{\hat0}^{}/\langle R^{\rW_\chi}_{\hat0,+}\rangle\otimes_{R_{\hat0}}\mM\is\mM\otimes_{R_{\hat0}}R_{\hat0}^{}/\langle R^{\rW_\chi}_{\hat0,+}\rangle\is\mM\otimes_{\U}\cZ_\lambda.\] 
It follows from the construction that those isomorphisms satisfy the required compatibility conditions and 
the data $(\cZ_\lambda,b_\mM)$
defines an element in $Z(\mathcal{HC}_{\hat\lambda},\otimes_{\U})$.

Proof of (2).
Let $\tilde\mE_\chi\in M_{\chi}$
be the image of $\mE_\chi$ under the equivalence 
$(i^{0})^{-1}:H_{\chi}\is M_{\chi}$ in Lemma \ref{M and H}.
Then by definition we have $\mathbf M(\mE_\chi)\is R\Gamma^{\hat\lambda,\widehat{-\lambda-2\rho}}(\pi^\circ\tilde\mE_\chi)$, 
where $\pi:Y\ra Y/T$.
Consider the map \[a:T\times (G/U\times G/U)/T\ra (G/U\times G/U)/T,\ (t,gU,g'U)\ra 
(gt^{-1}U,g'U).\] 
Then it follows from the definition of $(i^0)^{-1}$ that $\tilde\mE_\chi=a_*(\mE_\theta\boxtimes\Delta_*\mO_{G/B})$, here $\Delta:G/B\ra (G/U\times G/U)/T$ is the embedding $gB\ra (gU,gU)\on{mod}\ T$.
It is shown in \cite[Proposition 3.1]{BFO} that $R\Gamma(\Delta_*\mO_{G/B}\otimes p_2^*\omega_{G/B})\is\tU$
($p_2$ is the right projection map $(G/U\times G/U)/T\ra G/B$), hence by Lemma \ref{action} we get
\[
R\Gamma^{\hat\lambda,\widehat{-\lambda-2\rho}}(\pi^\circ\tilde\mE_\chi)\is
R\Gamma^{\hat\lambda}(\tilde\mE_\chi\otimes p_2^*\omega_{G/B})=R\Gamma^{\hat\lambda}(a_*(\mE_\chi\boxtimes(\Delta_*\mO_{G/B}\otimes p_2^*\omega_{G/B}))\is\]
\[\is R\Gamma(\Delta_*\mO_{G/B}\otimes p_2^*\omega_{G/B})\otimes_R^LR\Gamma^{\hat\lambda}(\mE_\chi)\is
\tU\otimes_{R} \Gamma^{\hat\lambda}(\mE_\chi).\]
Since $\Gamma^{\hat\lambda}(\mE_\chi)\is\dot{R}_\chi^{\lambda+\rho}$ for regular $\lambda$,
part (2) follows.

\end{proof}

\quash{
\subsubsection{}
Consider the following 
full 
abelian subcategory 
$H_{\chi,\chi}^t:=\mathbf M^{-1}(\mathcal{HC}_{\hat\lambda})\subset\sD(H_{\chi,\chi})$.
For any $\mM,\mM'\in H_{\chi,\chi}^t$
we define 
\[\mM*^t\mM':=\mathbf M^{-1}(\mathbf M(\mM)\otimes\mathbf M(\mM'))\in H_{\chi,\chi}^t.\]
One can check that $*^t$ defines a monoidal structure on $M_{\chi,\chi}^t$ and the functor $\mathrm M$ induces an
equivalence of abelian monoidal categories 
\beq\label{mon equ ab setting}
\mathbf M^t:(H_{\chi,\chi}^t,*^t)\is(\mathcal{HC}_{\hat\lambda},\otimes_{\U})
\eeq

\begin{corollary}\label{central structure}
We have $\mE_\chi\in H^t_{\chi,\chi}$. 
For every $\mM\in H^t_{\chi,\chi}$ there is a canonical isomorphism 
\[b_\mM:\mE_\chi*^t\mM\is\mM*^t\mE_\chi\] such that 
the data $(\mE_\chi,b_\mM)_{\mM\in H^t_{\chi,\chi}}$ defines 
an element in the Drinfeld center $Z(H^t_{\chi,\chi},*^t)$.

\end{corollary}
\begin{proof}
This follows immediately from above proposition and (\ref{mon equ ab setting}).
\end{proof}
}
We finish this section with a lemma to be used in the next section.
\begin{lemma}\label{tech}
Let $\lambda\in\check\ft$ be regular dominant weight and let $\mu$ be a dominant weight in 
$\lambda+\Lambda$.
We have $\theta_\lambda^\mu(\cZ_\lambda)\is\cZ_\mu$.
\end{lemma}
\begin{proof}
Note that $\theta_\lambda^\mu:\HC_{\hat\lambda}\to\HC_{\hat\mu}$ is monoidal, so we have 
$\theta_\lambda^\mu(\U_{\hat\lambda})\is\U_{\hat\mu}$. 
Note also that, by \cite[Proposition 1.4]{BG}, for any
$r\in R^{\rW_{\chi}}_{\hat0}$ and $\cM\in\HC_{\hat\lambda}$, we have 
$\theta_\lambda^\mu(b_\lambda(r)\cdot m)=b_\mu(r)\cdot\theta_{\lambda}^\mu(m)$,
where $m\in\cM$ and $b_\nu:R_{\hat0}^{\rW_{\mu+\rho}}\to Z(\U_{\hat\nu}),\ \nu\in\check\ft$ is the map in~\eqref{map b}\footnote{The group 
$\rW_\lambda$ in \cite{BG} is the group $\dot{\rW}_\lambda$ in the present paper}.
Since $\calZ_\lambda$ (resp. $\calZ_\mu$) is isomorphic to the coinvariant algebra of 
$\U_{\hat\lambda}$ (resp. $\U_{\hat\mu}$) with respect to the action of 
$R^{\rW_\chi}_{\hat0}$ via the map $b_\lambda$ (resp. $b_\mu$),
it follows that 
$\theta_{\lambda}^\mu(\calZ_\lambda)\is\calZ_\mu $.
 \end{proof}

\quash{
1)
The functor $\on{Av}_U^\xi$ maps $CS_\theta$ to $M_{\xi,\xi}^t$. 
We denote the resulting functor by $\on{Av}$
\\
2) 
There is a canonical equivalence 
\[{\on{av}}_U^\xi:CS_{\theta}\is Z(H^t_{\xi,\xi},*^t)\]
such that the composition 
$CS_{\theta}\is Z(H^t_{\xi,\xi},*^t)\ra M_{\xi,\xi}^t$
is isomorphic to $\on{Av}_U^\xi$.
}

\subsection{Proof of Theorem \ref{Key} in the de Rham setting}\label{proof in the case k=C}

Let $\cZ_\lambda\in Z(\mathcal{HC}_{\hat\lambda},\otimes_{\U})$, $\tilde\mE_\chi\in M_{\chi}$ 
be as 
in Proposition \ref{E=Z}.  
Define 
$\tilde\mE_\theta\is\bigoplus_{\chi\in\theta}\tilde\mE_\chi\in\sD(G\backslash Y)^\heartsuit$.
By the discussion above there exists a character $D$-module $\cM_\theta'\in CS_\theta$ such that 
\[R\Gamma^{\hat\lambda,\widehat{-\lambda-2\rho}}\circ\pi^\circ\circ\on{HC}(\cM_\theta')=\cZ_\lambda.\]
Hence by Proposition \ref{E=Z}, we have   
\[R\Gamma^{\hat\lambda,\widehat{-\lambda-2\rho}}\circ\pi^\circ\circ\mathrm{HC}(\mM_\theta')\is
R\Gamma^{\hat\lambda,\widehat{-\lambda-2\rho}}(\pi^\circ\tilde\mE_\chi)\]
for any regular dominant $\lambda\in\breve\ft$
mapping to $\xi$.
Since $\pi^\circ\circ\on{HC}:\sD(CS_\theta)\ra\bigoplus_{\chi\in\theta}\sD(M_{\chi,\chi^{-1}})$
and 
$R\Gamma^{\hat\lambda,\widehat{-\lambda-2\rho}}:
\sD(M_{\chi,\chi^{-1}})\is \mathcal{HC}_{\hat\lambda}$ is an equivalence of category for 
any regular dominant $\lambda$,
it follows that  \[\mathrm{HC}(\mM_\theta')\is\tilde\mE_\theta.\]
Applying the equivalence $i^\circ:\sD(M_\xi)\is\sD(H_\xi)$ and using 
Lemma \ref{M and H}, we get 
\beq\label{vanishing}
\on{Av}_*^U(\mM_\theta')\is i^\circ\circ\mathrm{HC}(\mM_\theta')\is i^\circ\tilde\mE_\theta\is\mE_\theta.
\eeq
The isomorphism above implies 
$\mE_\theta\is\on{Av}_*^U(\mM_\theta')\is\Res_{T\subset B}^G(\mM_\theta')$, therefore, 
by part $(3)$ of Proposition \ref{CS}, there is
a canonical $\rW$-equivaraint structure on $\mE_\theta$
such that $\cM'_\theta\is\Ind_{T\subset B}^G(\mE_\theta)^\rW$.
We claim that the above $\rW$-equivaraint structure on $\mE_\theta$ is isomorphic to the one
constructed in Section \ref{central Loc}.
Thus we have 
$\cM_\theta'\is\cM_\theta$. 
The theorem follows.

\subsection{Proof of the claim}
\subsubsection{}
We first give a description of~\eqref{vanishing} at the level of global sections (see~\eqref{equ 3} below).
 By Lemma \ref{action}, Example \ref{formula for Av}, and Theorem \ref{BFO}, 
the global sections of $\cM_\theta'$, $\on{Av}_*^U(\mM_\theta')$, and $\iota_*\mE_\theta$ (here $\iota:T=B/U\to G/U$) are given by
\[\Gamma^{}(\on{Av}_*^U(\mM_\theta'))\is (\U/\U\fn)\otimes_{\U}\Gamma(\cM'_\theta),\ \ \ \Gamma(\cM_\theta')\is\bigoplus_{\mu\in\lambda+\Lambda/\rW_\chi'}\theta_\lambda^\mu(\calZ_\lambda),\]
\[\Gamma(\iota_*\mE_\theta)=(\U/\U\fn)\otimes_{R}\Gamma(\mE_\theta).\]
By Lemma \ref{tech},
$\theta_\lambda^\mu(\cZ_\lambda)\is\cZ_\mu$, 
for any dominant  $\mu\in\lambda+\Lambda$.
Since every $\rW_\chi'$-orbit in $\lambda+\Lambda$ (under the dot action) contains a unique 
dominant weight, 
we have 
\beq\label{equ 0}
\Gamma(\cM_\theta')\is\bigoplus_{\mu\in\lambda+\Lambda/\rW_\chi'}\theta_\lambda^\mu(\calZ_\lambda)\is
\bigoplus_{\mu\in\lambda+\Lambda/\rW_\chi'}\calZ_\mu\is
\bigoplus_{\mu\in\lambda+\Lambda/\rW_\chi'}\U\otimes_Z\dot{R}_\chi^{\mu+\rho},
\eeq
which gives rise to
\beq\label{equ 1}
\Gamma(\Av_*^U(\cM_\theta'))\is(\U/\U\fn)\otimes_{\U}\Gamma(\cM_\theta')\is
(\U/\U\fn)\otimes_{\U}\bigoplus_{\mu\in\lambda+\Lambda/\rW_\chi'}\U\otimes_Z\dot{R}_\chi^{\mu+\rho}\is
\eeq
\[\is(\U/\U\fn)\otimes_R\bigoplus_{\mu\in\lambda+\Lambda/\rW_\chi'}R\otimes_Z\dot{R}_\chi^{\mu+\rho}.
\]
Since the map $\check\ft//\dot{\rW_\mu}\to\check\ft//\dot{\rW}$ is \'etale at 
the image of the $\dot\rW$-orbit $\dot{\rW}\mu\subset\check\ft$ along the projection $\check\ft\to\check\ft//\dot{\rW_\mu}$, it follows that
\[R\otimes_Z\dot{R}_\chi^{\mu+\rho}\is R\otimes_{R^{\dot\rW}}\dot{R}_\chi^{\mu+\rho}\is
\bigoplus_{\nu\in\dot{\rW}\mu,
\chi_\nu=\exp(\nu)} R\otimes_{R^{\dot{\rW}_\nu}}\dot{R}_{\chi_\nu}^{\nu+\rho}.\]
Since $R\otimes_{R^{\dot{\rW}_\nu}}\dot{R}_{\chi_\nu}^{\nu+\rho}\is
m_{-\nu}^*(R/\langle R_+^{\rW_{\chi_\nu}}\rangle)\is
\Gamma^{\hat\nu}(\mE_\theta)$, it implies
\beq\label{equ 2}
\bigoplus_{\mu\in\lambda+\Lambda/\rW_\chi'}
R\otimes_Z\dot{R}_\chi^{\mu+\rho}\is \bigoplus_{\mu\in\lambda+\Lambda/\rW_\chi'}\bigoplus_{\nu\in\rW\cdot\mu, \chi_\nu=\exp(\nu)}R\otimes_{R^{\dot{\rW}_\nu}}\dot{R}_{\chi_\nu}^{\nu+\rho}
\eeq
\[\is
\bigoplus_{\nu\in\dot{\rW}(\lambda+\Lambda)}\Gamma^{\hat\nu}(\mE_\theta)\is\Gamma(\mE_\theta).\]
Combining~\eqref{equ 1} and~\eqref{equ 2}, we obtain the 
description of~\eqref{vanishing} on global sections:
\beq\label{equ 3}
\Gamma(\Av_*^U(\cM_\theta'))\is(\U/\U\fn)\otimes_R\bigoplus_{\mu\in\lambda+\Lambda/\rW_\chi'}R\otimes_Z\dot{R}_\chi^{\mu+\rho}\is
(\U/\U\fn)\otimes_R\Gamma(\mE_\theta)\is\Gamma(\iota_*\mE_\theta).
\eeq
\begin{remark}\label{descent to t/W}
Note that if we identify $\dot\rW=\rW_{a,-\rho}$ as subgroup in $\rW_a$, then
$Z=R^{\dot\rW}=R^{\rW_{a,-\rho}}$ and the 
isomorphism~\eqref{equ 2} 
restricts to an isomorphism 
\beq\label{W_a-inv}
\bigoplus_{\mu\in\lambda+\Lambda/\rW_\chi'}\dot{R}_{\chi}^{\mu+\rho}\is
(R^{\rW_{a,-\rho}}\otimes_{R^{\rW_{a,-\rho}}}\bigoplus_{\mu\in\lambda+\Lambda/\rW_\chi'}\dot{R}_{\chi}^{\mu+\rho})\is
(R\otimes_{R^{\rW_{a,-\rho}}}\bigoplus_{\mu\in\lambda+\Lambda/\rW_\chi'}\dot{R}_{\chi}^{\mu+\rho})^{\rW_{a,-\rho}}\is
\Gamma(\mE)^{\rW_{a,-\rho}}
\eeq
such that the map
\beq\label{descent map}
R\otimes_{R^{\rW_{a,-\rho}}}\Gamma(\mE)^{\rW_{a,-\rho}}\stackrel{~\eqref{W_a-inv}}\is
R\otimes_{R^{\rW_{a,-\rho}}}\bigoplus_{\mu\in\lambda+\Lambda/\rW_\chi'}\dot{R}_{\chi}^{\mu+\rho}\stackrel{~\eqref{equ 2}}\is\Gamma(\mE)
\eeq
is the isomorphism
in Theorem \ref{characterization} (4).
\end{remark}
\subsubsection{}

Let $x\in N(T)$ and $w\in\rW$ its image in the Weyl group.
Let $U_w=xUx^{-1}$ with Lie algebra $\fn_w$.
Consider the following commutative diagram 
\beq\label{Ad_x}
\xymatrix{G\ar[r]^\pi\ar[d]^{\Ad_x}&G/U\ar[d]^{\Ad_x}&T\ar[l]_{\ \ \iota}\ar[d]^{w}\\
G\ar[r]^{\pi_w}&G/U_w&T\ar[l]_{\ \ \iota_w}}
\eeq
where the horizontal maps are the natural quotient maps or embeddings.
Let 
\[\phi_w:\mE_\theta\is w^*\mE_\theta\]
\[\phi_x:\mM'_\theta\is\Ad_x^*\mM'_\theta\] 
be the isomorphisms coming from 
the $G$-conjugation equivariant structures on $\mM'_\theta$
and the $\rW$-equivariant structure on $\mE_\theta$, and let 
\[\Gamma(\phi_w):\Gamma(\mE)\to\Gamma(w^*\mE)\is\Gamma(\mE)\]
\[\Gamma(\phi_x):\Gamma(\mM_\theta')\to\Gamma(\Ad_x^*\mM_\theta')\is\Gamma(\mM_\theta')\]
be the induced maps on global sections\footnote{Recall that there are canonical identifications 
$\Gamma(\Ad_x^*\mM_\theta')\is\Gamma(\mM_\theta')$,
$\Gamma(w^*\mE_\theta)\is\Gamma(\mE_\theta)$}.
Then the diagram~\eqref{Ad_x} gives rise to natural isomorphisms
\[\iota_*\phi_w:\iota_*\mE\is\iota_*w^*\mE\is\Ad_x^*(\iota_{w})_*\mE,\]
\[\Av_*^U(\phi_x):\Av_*^U\mM'_\theta\is\Av_*^U\Ad_x^*\mM_\theta'\is\Ad_x^*\Av^{U_w}_*\mM_\theta',\]
such that the induced maps on global sections are given by 
\beq\label{iota_w}
\Ad_x(-)\otimes\Gamma(\phi_w):\Gamma(\iota_*\mE)=\U/\U\fn\otimes_S\Gamma(\mE)\lra
\Gamma(\Ad_x^*(\iota_{w})_*\mE)=\U/\U\fn_w\otimes_R\Gamma(\mE)
\eeq
\beq\label{iota_x}
\Ad_x(-)\otimes\Gamma(\phi_x):\Gamma(\Av_*^U\mM_\theta')=\U/\U\fn\otimes_{\U}\Gamma(\mM_\theta')\to\Gamma(\Ad_x^*\Av^{U_w}_*\mM_\theta')=\U/\U\fn_w\otimes_{\U}\Gamma(\mM_\theta').
\eeq
To prove the claim,  we need to show that the following diagram commute
\beq\label{key diagram}
\xymatrix{
\Gamma(\Av_*^U\mM_\theta')\ar[r]^{\eqref{iota_x}}\ar[d]_{\eqref{equ 3}}&\Gamma(\Av_*^{U_w}\mM_\theta')\ar[d]^{\eqref{equ 3}}\\
\Gamma(\iota_*\mE)\ar[r]^{\eqref{iota_w}}&\Gamma((\iota_w)_*\mE).
}
\eeq
where the right vertical arrow 
is the 
map~\eqref{equ 3} in the case when the unipotent radical is $U_w$.
\subsubsection{}
By Theorem \ref{BFO},
isomorphism~\eqref{equ 0} intertwines the map
$\Gamma(\phi_x):\Gamma(\cM_\theta')\to\Gamma(\cM_\theta')$ with
\beq\label{Av 2}
\Ad_x(-)\otimes\on{id}:\U\otimes_Z\bigoplus_{\mu\in\lambda+\Lambda/\rW_\chi'}\dot{R}_\chi^{\mu+\rho}\stackrel{}\lra\U\otimes_Z\bigoplus_{\mu\in\lambda+\Lambda/\rW_\chi'}\dot{R}_\chi^{\mu+\rho}.
\eeq
Using Remark \ref{descent to t/W}, it follows that we have the following 
commutative diagram
\beq\label{key diagram 2}
\xymatrix{
\Gamma(\Av_*^U\mM_\theta')\ar[r]^{\eqref{iota_x}}\ar[d]_{\eqref{equ 3}}&\Gamma(\Av_*^{U_w}\mM_\theta')\ar[d]^{\eqref{equ 3}}\\
\Gamma(\iota_*\mE)\ar[r]^{\eqref{Av formula}}&\Gamma((\iota_w)_*\mE).
}
\eeq
where the map \eqref{Av formula} is given by
\beq\label{Av formula}
\xymatrix{\Gamma(\iota_*\mE)=(\U/\U\fn)\otimes_R\Gamma(\mE)\ar[r]^{~\eqref{Av formula}}\ar[d]^{\eqref{descent map}}&\Gamma((\iota_w)_*\mE)=(\U/\U\fn_w)\otimes_R\Gamma(\mE)\\
(\U/\U\fn)\otimes_RR\otimes_{R^{\rW_{a,-\rho}}}\Gamma(\mE_\theta)^{\rW_{a,-\rho}}\ar[r]&
(\U/\U\fn_w)\otimes_RR\otimes_{R^{\rW_{a,-w(\rho)}}}\Gamma(\mE_\theta)^{\rW_{a,-w(\rho)}}\ar[u]^{\eqref{descent map}}}
\eeq
where the bottom arrow is the map 
$\Ad_x(-)\otimes w\otimes\Gamma(\phi_w)$
\footnote{Note that the map is well-defined since 
$w\rW_{a,-\rho}w^{-1}=\rW_{a,-w(\rho)}$}.

On the other hand, by Theorem \ref{characterization}, we have the following commutative diagram
\[\xymatrix{\Gamma(\mE)\ar[r]^{\Gamma(\phi_w)}\ar[d]&\Gamma(\mE)\ar[d]\\
R\otimes_{R^{\rW_{a,-\rho}}}\Gamma(\mE_\theta)^{\rW_{a,-\rho}}\ar[r]^{w\otimes\Gamma(\phi_w)}&R\otimes_{R^{\rW_{a,-w(\rho)}}}\Gamma(\mE_\theta)^{\rW_{a,-w(\rho)}}},\]
were the vertical maps are the isomorphisms in Theorem \ref{characterization} (4).
Thus the map~\eqref{Av formula} is equal to~\eqref{iota_w}, and hence 
\eqref{key diagram 2} is equal to~\eqref{key diagram}.
The commutativity of~\eqref{key diagram} follows. 
This finishes the proof of the claim.

\subsection{Mixed characteristic lifting of $\mE_\theta$ and $\cM_\theta$ }\label{mixed char}
Let $G_\bZ$ be a split reductive group over $\bZ$.
Let $T_\bZ$ be a maximal torus of $G_\bZ$
and let $B_\bZ$ be a Borel subgroup containing $T_\bZ$
with unipotent radical $U_\bZ$.
For any ring $R$ (resp. any scheme $S$), we denote by $G_R$, $B_R$, etc,
(resp. $G_S$, $B_S$, etc,) the base change of 
$G_\bZ$, $B_\bZ$, etc, along $\Spec(R)\to\Spec(\bZ)$ (resp. $S\to\Spec(\bZ)$).

Let $A\subset\bC$ be a strict Henselization
of a closed point of $\bZ[1/\ell]$ with
residue field $k$.
Let $\chi\in\calC(T_k)(\barQ)$ and $\theta=\rW\chi$ be the $\rW$-orbit of 
$\chi$. 
Let $\mE_\theta$ be the $\rW$-equivaraint $\ell$-adic local system on $T_k$
in Section \ref{central Loc}.

\begin{lemma}\label{lifting of E_theta}
There exists a $\rW$-equivaraint $\ell$-adic local system $\mE_{\theta,A}$ on $T_A$
which, after base change $A\to k$, becomes 
$\mE_\theta$.
\end{lemma}
\begin{proof}
Let $\rho_\theta$ be the $\ell$-adic representaiton of 
$\rW\ltimes\pi_1^t(T_k)$ associated to $\mE_\theta$.
The specialization isomorphism $\on{sp}:\pi_1^t(T_k)\to\pi_1^t(T_A)$
(see \cite{Gr})
induces an isomorphism $\rW\ltimes\pi_1^t(T_k)\is\rW\ltimes\pi_1^t(T_A)$, thus we 
can indentify $\rho_\theta$ as a $\ell$-adic representaiton of $\rW\ltimes\pi_1^t(T_A)$
and we denote by $\mE_{\theta,A}$ the corresponding local sysetm on $T_A$.
It is obvious that $\mE_{\theta,A}$ satisfies the desired property.
\end{proof}

Let $\mM_\theta$ be the character sheaf associated to $\theta$. Our next goal is construct
a mixed characteristic version of $\mM_\theta$.
For this we observe 
that for any flat group scheme $\mG$ finite type over $\bZ$ and a 
closed subgroup scheme $\mH\subset\mG$ flat over $\bZ$, the universal
geometric quotient $\mG/\mH$ exists \cite{A}.
Note also that the Chevalley  isomorphism holds 
$G_R//G_R\is T_R//W_R$ 
for any ring $R$ and the formation commutes with arbitrary base change 
$R\to R'$ \cite{Le}.
It follows that the quotient map
$\pi:G\to G/U$ and
the Grothendieck-Springer simultaneous resolution in~\eqref{G-S resolution}
makes sense over any ring $R$. 
Moreover, the formation commutes with 
arbitrary base change $R\to R'$. 
We denote them by
$\pi_R:G_R\to G_R/U_R$ and
\beq\label{G-S resolution over R}
\xymatrix{\widetilde G_R\ar[r]^{\tilde q_R}\ar[d]^{\tilde c_R\ \ }&T_R\ar[d]^{q_R}\\
G_R\ar[r]^{c_R}&T_R//\rW_R}
\eeq

Denote by $h_R:\widetilde G_R\to Z_R:=G_R\times_{T_R//\rW}T_R$ the induced map.
We define 
\beq\label{IC}
\IC(Z_R):=(h_Z)_!\barQ[\dim G_\bC]\in\sD(Z_R).
\eeq
When $R=k$ is an algebraically closed field of characteristic not equal to $\ell$, 
$\IC(Z_k)$ is the IC-complex of 
$Z_k$ and there is a canonical $\rW$-equivaraint structure
$(\IC(Z_k),\phi_k)\in\sD_\rW(Z_k)$ (see Section \ref{W-action}).

\begin{lemma}\label{lifting of W action}
There exists a positive integer $N$, depends on the group $G_\bZ$,
satisfies the following.
Let $A$ be the strict Henselization of $\bZ[1/N\ell]$ at a closed point with residue field $k$.
One can endow
the $\ell$-adic complex
$\IC(S_A)$ in~\eqref{IC} with
a $\rW$-equivariant structure $(\IC(Z_A),\phi_A)\in\sD_\rW(Z_A)$ which, under the base change 
$A\to k$, becomes $(\IC(Z_k),\phi_k)\in\sD_\rW(S_k)$.

\end{lemma}
\begin{proof}
According to \cite[Section 6.1]{BBD} (or the discussion in \cite[Section 4]{Dr}), one can choose a (large enough) positive integer 
$N$, a stratification $\calT$ of $S_{\bZ[1/N\ell]}$, and for each $T\in\calT$
a (finite) collection $\calL(T)$ of 
$\ell$-adic local systems on $T$, satisfy the following:
(1) We have $w^*\IC(Z_{\bZ[1/N\ell]})\in\sD_{\mT,\mL}(Z_{\bZ[1/N\ell]})$ for $w\in\rW$,
(2) Let $A,k$ be as above. Let $i:Z_k\to Z_A$ be the imbedding.
The functor $i^*:\sD_{\calT,\calL}(Z_A)\to\sD_{\calT,\calL}(Z_k)$
 is an equivalence. Here $\sD_{\calT,\calL}(Z_A)$ (resp. $\sD_{\calT,\calL}(Z_k)$)
 is the full subcategory of $\sD(Z_A)$ (resp. $\sD(Z_k)$) generated by 
 the $*$-restriction of $\calL(T), T\in\calT$ to $Z_A$ (resp. $Z_k$).

Note that, by (1) above, we have 
$w^*\IC(Z_A)\in\sD_{\calT,\calL}(Z_A)$ and $w^*\IC(Z_k)\in\sD_{\calT,\calL}(Z_k)$.
Let $\phi_{k,w}:\IC(S_k)\is w^*\IC(Z_k)$ be the isomorphism coming 
from the $\rW$-equivariant structure $
\phi_k$ on $\IC(Z_k)$.
Since $i^*\IC(Z_A)\is\IC(Z_k)$, it follows from (2) above that 
there exists an isomorphism 
$\phi_{A,w}:\IC(Z_A)\is w^*\IC(Z_A)$ which, under the base change $A\to k$,
becomes $\phi_{k,w}$. Now the collection $\{\phi_{A,w}|w\in\rW\}$
defines a $\rW$-equivariant structure $\phi_A$ on $\IC(Z_A)$ satisfying the 
required property.
\end{proof}

\begin{prop}\label{lifting of M_theta}
There exists a positive integer $N$, depends on the group $G_\bZ$,
satisfies the following:
Let $A$ be the strict Henselization of $\bZ[1/N\ell]$ at a closed point with residue field $k$.
There is a $\ell$-adic complex 
$\cM_{\theta,A}$ on $G_A$ which, under the base change $A\to k$, becomes 
$\cM_\theta$.

\end{prop}
\begin{proof}
Let $N,A,k$ be as in Lemma \ref{lifting of W action}.
Let $\mE_{\theta,A}\in\sD_\rW(T_A)$ be the lift of $\mE_\theta$ in Lemma \ref{lifting of E_theta}. Define 
\[\Ind_{T_A\subset B_A}^{G_A}(\mE_{\theta,A}):=(\tilde c_A)_!(\tilde q_A)^*(\mE_{\theta,A})[\dim G_\bC-\dim T_\bC]\is\]
\[\is p_{G_A,!}(p_{T_A}^*(\mE_{\theta,A})\otimes\IC(Z_A))[\dim G_\bC-\dim T_\bC].\]
Here $p_{T,A}$ and $p_{G,A}$ are the natural projections from
$Z_A$ to $T_A$ and $G_A$ respectively.
The $\rW$-equivaraint structures on $\mE_{\theta,A}$ and $\IC(Z_A)$
give rise to a $\rW$-action on $\Ind_{T_A\subset B_A}^{G_A}(\mE_{\theta,A})$
and we define 
$\cM_{\theta,A}=\Ind_{T_A\subset B_A}^{G_A}(\mE_{\theta,A})^{\rW}$ to be the $\rW$-invariant factor.
Since the base change of $\mE_{\theta,A}$ along $A\to k$
is isomorphic to $\mE_{\theta}$
, it follows that the
base change of  
$\cM_{\theta,A}$ along $A\to k$ is isomorphic to $\cM_{\theta}=\Ind_{T_k\subset B_k}^{G_k}(\mE_\theta)^\rW$.

\end{proof}

\subsection{ULA property of the averaging functor}
We first review the notion of universal local acyclicity (ULA) 
following \cite{De2,Z}.

Let $S$ be a Noetherian scheme.
Let $f:X\to S$ be a morphism of finite type
and let $\mF\in\sD(X)$. 
Let $s$ be a geometric point of $S$ and let
$S_{(s)}$ be the strict Henselisation at $s$.
We recall the following definition in \cite{De2}:
\begin{definition}
A $\ell$-adic complex $\mF\in\sD(X)$ is called locally acyclic with respect to
$f:X\to S$ if for every geometric point
$x\in X$ and every geometric point $t\in S_{(f(x))}$, the natural map
$\oH^*(X_{(x)},\mF)\to \oH^*((X_{(x)})_t,\mF)$ is an isomorphism,
where $(X_{(x)})_t=(X_{(x)})\times_{S_{(f(x))}}t$.
It is called universally locally acyclic (ULA) if 
it it locally acyclic after arbitrary base change $S'\to S$.
\end{definition}

One can reformulate local acyclicity as follows.
Let $t$ be a geometric point of 
$S_{(s)}$. Denote by $j_t:X\times_St\to X$ and $i_s:X\times_Ss\to X$
the natural maps.
We write 
\beq\label{Psi}
\Psi_{t\to s}(\mF):=i_s^*(j_t)_*j_t^*(\mF).
\eeq
It is shown in \cite[Lemma A 2.2]{Z} that $\mF$ is locally acyclic with respect to
$f$
if and only if the natural map
\[i_s^*\mF\to \Psi_{t\to s}(\mF)\] 
is an isomorphism.

Here are some properties of ULA complexes that we need.

\begin{thm}\label{ULA}
\begin{enumerate}
\item Let $f:X\to Y$ be a proper morphism over $S$, and let 
$\mF$ be a $\ell$-adic complex on $X$ ULA with respect to $X\to S$.
Then $f_!\mF$ is ULA with respect to $Y\to S$.
\item
Let $f:X\to Y$ be a smooth morphism over $S$, and let 
$\mF$ be a $\ell$-adic complex on $Y$. Then
$\mF$ is
 ULA with respect to $Y\to S$ if and only if 
$f^*\mF$ is ULA with respect to $X\to S$.
\item
Let $f_i:X\to S, i=1,2$ and let $\mF_i$ be a $\ell$-adic complex on $X_i$ ULA with respect to
$X_i\to S$. Then $\mF_1\boxtimes_S\mF_2$ is ULA with respect to
$X_1\times_SX_2\to S$.
\item 
Let $f:X\to S$ be a morphism of finite type and let 
$\mF$ be a $\ell$-adic complex on $X$.
Then there is an open dense subset $U$ of $S$ such that 
the restriction of $\mF$ to $X_U=X\times_SU$ is ULA with respect to
$X_U\to U$.
\item
Let $f:X\to S$ be a smooth morphism.
Then any $\ell$-adic local system $\mF$ on 
$X$ is ULA with respect to $f:X\to S$.

\end{enumerate}
\end{thm}

Consider the open embedding $j_S:U_S\times_S\bar B_S\to G_S$,
where $\bar B_S$ is the opposite Borel. 
Recall the quotient map
$\pi_S:G_S\to G_S/U_S$.
We have the following ULA property of the averaging functor.

\begin{prop}\label{ULA for Av}
Assume $j_{S,*}(\barQ)$ is ULA with respect to $G_S\to S$, of
formation compatible with arbitrary change of base on $S$.
Let
$\mF$ be a $\ell$-adic complex on $G_S$ ULA with respect to 
$G_S\to S$. Then $\pi_{S,*}(\mF)$ is ULA with respect to 
$G_S/U_S\to S$, of
formation compatible with arbitrary change of base on $S$.
\end{prop}
\begin{proof}
Consider the following Cartesian diagram
\[\xymatrix{U_S\times_S G_S\ar[r]^{a}\ar[d]^\pr&G_S\ar[d]^{\pi_S}\\
G_S\ar[r]^{\pi_S}&G_S/U_S}\]
where $\pr$ is the projection map and $a$ is the left action map.
As $\pi_S$ is smooth surjective morphism, it suffices to show that
\[(\pi_S)^*\pi_{S_*}(\mF)\is a_*\pr^*(\mF)\is a_*(\bar\bQ\boxtimes_S\mF)\]
is ULA w.r.t $G_S\to S$, of
formation compatible with arbitrary change of base on $S$.
Note that $a$ admits the following factorization
\[U_S\times_S G_S\stackrel{h}\to G_S\times^{\bar B_S}_SG_S\stackrel{\bar a}\to G_S\]
where $h$ is the open embedding 
$U_S\times_S G_S\is (U_S\times_S \bar B_S)\times_S^{\bar B_S} G_S\subset
G_S\times^{\bar B_S}_SG_S
$ and $\bar a$ is the left action map.
As $\bar a$ is proper, to show that 
\[a_*(\bar\bQ\boxtimes_S\mF)\is
\bar a_*h_*(\bar\bQ\boxtimes_S\mF)\]
is ULA w.r.t $G_S\to S$, of
formation compatible with arbitrary change of base on $S$, it suffices to show that 
\[h_*(\bar\bQ_{}\boxtimes_S\mF)\]
is ULA w.r.t $G_S\times^{\bar B_S} G_S\to S$, of
formation compatible with arbitrary change of base on $S$.
Consider the following Cartesian diagram
\[\xymatrix{(U_S\times_S\bar B_S)\times_S G_S\ar[r]^{\ \ \ j_S\times\id}\ar[d]&G_S\times_S G_S\ar[d]^{q}\\
U_S\times_S G_S\ar[r]^h&G_S\times^{\bar B_S}_S G_S}.\]
where $j_S:U_S\times_S\bar B_S\to G_S$ is the open embedding.
Since $q$ is a smooth surjective morphism and
$j_{S,!}(\bar\bQ_\ell)$ and $\mF$ are ULA w.r.t $G_S\to S$ (by assumption),  we have 
\[q^*h_*(\bar\bQ\boxtimes_S\mF)\is j_{S,*}(\bar\bQ_\ell)\boxtimes_S\mF,\]
which is ULA w.r.t $G_S\times_SG_S\to S$, of
formation compatible with arbitrary change of base on $S$.
The lemma follows

\end{proof}

\begin{remark}\label{generic base change}
By
Theorem \ref{ULA} (4) and Deligne's generic base change theorem \cite[Corollary 2.9]{De2}, 
the assumptions on $j_{S,*}(\barQ)$ holds 
after a base change to 
an open dense subset $U\subset S$.

\end{remark}

Let $A$ be a 
strictly Henselian local ring and
let $S=\Spec A$. Let $s$ be the closed point of $S$ and let $t$ be a geometric point of 
$S$.
Let $f:X\to S$ be a scheme over $S$ and let $Y\subset X$ be an open subscheme 
over $S$.

\begin{lemma}\label{generic zero}
Let $\mF$ be a $\ell$-adic complex on $X$ ULA with respect to $X\to S$.
Let $\mF_s$ and $\mF_t$ be the restriction of $\mF$ to the 
fiber $X_s$ and $X_t$ respectively.
Then $\mF_{t}|_{Y_t}\is 0$
implies $\mF_s|_{Y_s}\is 0$.
 \end{lemma}
 \begin{proof}
 Indeed, since $Y\to X$ is smooth
 and the functor $\Psi_{t\to s}$ in~\eqref{Psi} commutes with smooth pull back,
 the isomorphism $\mF_s\is\Psi_{t\to s}(\mF)$ (coming from the ULA property) implies 
$\mF_s|_{Y_s}\is \Psi_{t\to s}(\mF)|_{Y_s}\is 
 \Psi_{t\to s}(\mF|_Y)\is i_s^*(j_t)_*(\mF_t|_{Y_t})\is 0
 $.
 
\end{proof}

\subsection{Proof of Theorem \ref{Key} in the $\ell$-adic setting}
We shall show that there exists a positive integer $N$, depending only on $G_\bZ$, such
that for any
algebraically closed field $k$ of positive characteristic not dividing $N\ell$
and a $\rW$-orbit $\theta=\rW\chi$ of a tame character $\chi\in\calC(T_k)(\barQ)$, 
the
averaging $\Av_*^{U_k}(\cM_{\theta})$
is supported on 
$T_k=B_k/U_k\subset G_k/U_k$. Equivalently, 
the restriction of $\Av_*^{U_k}(\cM_{\theta})$ to the open complement 
$Y_k=(G_k/U_k)\setminus(B_k/U_k)$ is zero.

Let $N,A,\cM_{\theta,A},\cE_{\theta,A}$ be as in Proposition \ref{lifting of M_theta} such that 
$k$ is the residue field of $A$.
According to Remark \ref{generic base change}, by replacing $N$ with a larger positive integer,
we can assume 
$(j_{\Spec(\bZ[1/N])})_*\barQ$ is ULA w.r.t $G_{\bZ[1/N]}\to 
\Spec(\bZ[1/N])$, of formation compatible with arbitrary base change on $\Spec(\bZ[1/N])$.

We will write $\mE_{\theta,A'}$, $\calM_{\theta,A'}$ 
for the base change of $\mE_{\theta,A}$, $\calM_{\theta,A}$
along $A\to A'$.
Note that, by Lemma \ref{lifting of E_theta} and Proposition \ref{lifting of M_theta}, we have 
$\mE_{\theta,k}\is\mE_{\theta}$ and $\cM_{\theta,k}\is\cM_{\theta}$.

\begin{lemma}\label{ULA for M_A}
$(\pi_A)_*\cM_{\theta,A}$ is ULA with respect to $G_A/U_A\to\Spec(A)$,
of formation compatible with arbitrary base change on $\Spec(A)$.
\end{lemma}
\begin{proof}
Since the map $\tilde c_A$ (resp. $\tilde q_A$) in~\eqref{G-S resolution over R} \
is proper (resp. smooth) and $\mE_{\theta,A}$ is a $\ell$-local system, 
by Theorem \ref{ULA}, the induction 
$\Ind_{T_A\subset B_A}^{G_A}(\mE_{\theta,A}):=(\tilde c_A)_!(\tilde q_A)^*(\mE_{\theta,A})[\dim G_\bC-\dim T_\bC]$
is ULA w.r.t $G_A\to\Spec(A)$.
As $\cM_{\theta,A}$ is the $\rW$-invariant direct factor of $\Ind_{T_A\subset B_A}^{G_A}(\mE_{\theta,A})$
it implies $\cM_{\theta,A}$ is ULA w.r.t $G_A\to\Spec(A)$, of formation compatible with arbitrary base change on $\Spec(A)$.
By 
Proposition \ref{ULA for Av}, we conclude that $(\pi_A)_*\cM_{\theta,A}$ is ULA with respect to $G_A/U_A\to\Spec(A)$,
of formation compatible with arbitrary base change on $\Spec(A)$.
\end{proof}

\begin{lemma}\label{generic vanishing}
$(\pi_\bC)_*\cM_{\theta,\bC}$ is supported on $T_\bC=B_\bC/U_\bC\subset G_\bC/U_\bC$.
\end{lemma}
\begin{proof}
Let $D_c^b(G_\bC,\barQ)$ be the bounded derived category of constructible $\ell$-adic complexes on $G_\bC$ and let 
$D^b_c(G_\bC(\bC),\bC)$ be the usual bounded derived category of 
$\bC$-constructuble complexes on the complex Lie group $G_\bC(\bC)$.
We fix an isomorphism $\iota:\barQ\is\bC$. Then according to \cite[Section 6.1]{BBD}, there 
is a 
comparison functor
\[\epsilon^*:D_c^b(G_\bC,\barQ)\to D^b_c(G_\bC(\bC),\bC)\] which is
fully-faithful and commutes with six functor formalism.
Let $\mE_{\theta,\bC}$ be the base change of $\mE_{\theta,A}$ along $A\to\bC$.
We claim that there exists a 
character $\chi_\bC$ of the topological fundamental group 
$\pi_1(T(\bC))$ such that, under the Riemann-Hilbert correspondence (R.H. for short), 
$\epsilon^*\mE_{\theta,\bC}$ corresponds to the de Rham local system 
$\mE_{\theta_\bC}$ in Section \ref{central Loc}.
Here $\theta_\bC=\rW\chi_\bC$ is the $\rW$-orbit of $\chi_\bC$.
This will imply 
\[\epsilon^*\cM_{\theta,\bC}\is
\epsilon^*((\Ind_{T_\bC\subset B_\bC}^{G_\bC}(\mE_{\theta,\bC}))^{\rW})\is
(\Ind_{T_\bC\subset B_\bC}^{G_\bC}(\epsilon^*\mE_{\theta,\bC}))^\rW\stackrel{\on{R.H.}}\is
\Ind_{T_\bC\subset B_\bC}^{G_\bC}(\mE_{\theta_\bC})^\rW=\cM_{\theta_\bC}
\] 
and the lemma follows from Theorem \ref{Key} in the de Rham setting.

To prove the claim,
we observe that the constructions of $\mE_\theta$ and 
$\mE_{\theta,A}$ in Section \ref{central Loc} and
Lemma \ref{lifting of E_theta} imply that 
$\mE_{\theta,\bC}$ corresponds to a  $\ell$-adic representation 
$\Ind_{\rW_{\chi_{\bC,\ell}}}^\rW(\rho_{\bC,\ell}^{un}\otimes\chi_{\bC,\ell})$ where 
$\chi_{\bC,\ell}$ is a $\ell$-adic character of $\pi_1(T_\bC)$ with
$\rW_{\chi_{\bC,\ell}}=\rW_\chi$, and $\rho^{uni}_{\bC,\ell}$ is the representation 
of $\rW'_{\chi_{\bC,\ell}}\ltimes\pi_1(T_\bC)$ in 
$\barQ[[\pi_1(T_\bC)_\ell]]/\langle\barQ[[\pi_1(T_\bC)_\ell]]^{\rW_{\chi_{\bC,\ell}}}_+\rangle$
given by the $\barQ[[\pi_1(T_\bC)_\ell]]$-module structure. 
Here $\barQ[[\pi_1(T_\bC)_\ell]]$ is the completed group
algebra of the pro-$\ell$ part of the \'etale fundamental group $\pi_1(T_\bC)$.
Note that the restriction of the functor $\epsilon^*$ to the subcategory of $\ell$-adic local systems on $T_\bC$ is induced by 
the natural embedding 
\beq\label{completion}
\pi_1(T(\bC))\to\pi_1(T(\bC))_\ell\stackrel{}\is\pi_1(T_\bC)_\ell.
\eeq
Note also that~\eqref{completion}
induces an
isomorphism \[\on{S}/\langle\on{S}_+^{\rW_{\chi_\bC}}\rangle
\is\barQ[[\pi_1(T_\bC)_\ell]]/\langle\barQ[[\pi_1(T_\bC)_\ell]]^{\rW_{\chi_{\bC,\ell}}}_+\rangle\] 
compatible with the $\rW_{\chi_\bC}'=\rW'_{\chi_{\bC,\ell}}$-action, 
here $\on{S}=\bC[[\pi_1(T(\bC))]]$ is
the completion of the group algebra 
 $\bC[\pi_1(T(\bC))]$ at $1\in\pi_1(T(\bC))$.
Let $\rho^{uni}_\bC$ and $\chi_\bC$ be the representations of the topological fundamental group 
$\pi_1(T(\bC))$ given by pull back of  
$\rho^{uni}_{\bC,\ell}$ and $\chi_{\bC,\ell}$ along~\eqref{completion}.
It follows that
 $\rho_\bC^{uni}$ is isomorphic to the representation
 $\rho_{\chi_\bC}^{uni}$ in Section \ref{central Loc} (in the de Rham setting)
and hence the pull back of the representation
$\Ind_{\rW_{\chi_{\bC,\ell}}}^\rW(\rho_{\bC,\ell}^{uni}\otimes\chi_{\bC,\ell}))$
along~\eqref{completion} is isomorphic to
$
\Ind_{\rW_{\chi_\bC}}^\rW(\rho_{\chi_\bC}^{uni}\otimes\chi_{\bC})$.
Since, by construction, $\mE_{\theta_\bC}$ corresponds to 
$
\Ind_{\rW_{\chi_\bC}}^\rW(\rho_{\chi_\bC}^{uni}\otimes\chi_{\bC})$
under the Reimann-Hilbert correspondence, we conclude that  
$\epsilon^*\mE_{\theta,\bC}\stackrel{\on{R.H.}}\is\mE_{\theta_\bC}.$
The claim follows.

\end{proof}

Applying Lemma \ref{generic zero} to the case $\mF=(\pi_A)_*\cM_{\theta,A}$,
$X=G_A/U_A$,
$Y=Y_S=(G_A/U_A)\setminus (B_A/U_A)$, and using Lemma \ref{ULA for M_A} and Lemma \ref{generic vanishing},
we 
conclude that $\Av_*^{U_k}(\cM_\theta)\is(\pi_k)_*\cM_{\theta,k}$ is supported on $T_k$.
This finished the proof of Theorem \ref{Key} in the $\ell$-adic setting.

\section{Proof of Theorem \ref{main result}}\label{Proof}
In this section we prove the vanishing conjecture (Conjecture \ref{vanishing conj}) for strongly central complex (resp. strongly $*$-central complex):

\begin{thm}\label{main}
There exists a positive integer $N$ depending only on the type of the group $G$ such that 
the following holds. 
Assume $k=\bC$ or $\on{char}k=p$ is not dividing
$N\ell$. Let $\mF\in\sD_\rW(T)$ be a strongly central complex (resp. strongly $*$-central complex) on $T$
and let $\Phi_\mF=\Ind_{T\subset B}^G(\mF)^\rW\in\sD(G)$.
For any $x\in G\setminus B$, 
we have the following cohomology vanishing
\[
\ \ \ \ \ \ \ \ \ \oH_c^*(xU,i^*\Phi_{\mF})=0\ \ \ \ \ \ (\text{resp.}\ \ \oH^*(xU,i^!\Phi_{\mF})=0)
\]
where $i:xU\to G$ is the natural inclusion map.
Equivalently, 
$\Av^U_!(
\Phi_\mF)$ (resp. $\Av_*^U(\Phi_\mF)$) is supported 
on the closed subset $T=B/U\subset G/U$.

In particular, the vanishing conjecture holds for reductive groups with connected center.

\end{thm}

\subsection{Reduction to perverse sheaves}
\begin{lemma}\label{reduction}
If Conjecture \ref{vanishing conj} holds for central perverse sheaves (resp. $*$-central perverse sheaves), then 
it holds for arbitrary central complexes (resp. $*$-central complex).
The same is true for strongly central complexes (resp. strongly $*$-central complexes).
\end{lemma}
\begin{proof}
It is enough to verify the lemma for $*$-central complexes and strongly $*$-central complexes.
Let $\mF$ be a $*$-central complex.

The $\ell$-adic setting.
We shall show that 
$^p\tau_{\leq b}(\mF)$ and $^p\sH^b(\mF)$, $b\in\bZ$, are $*$-central. 
Let $\chi\in\calC(T)(\barQ)$ and let $I_\chi$ be the maximal ideal corresponding to $\chi$.
Write $\calC(T)_{\hat\chi}$
be the completion of at $\chi$.
Let $q_{\hat\chi}:\calC(T)_{\hat\chi}\to\calC(T)_{\hat\chi}//\rW_\chi$ be the quotient map.
Since 
$\chi$ is the unique closed point of $\calC(T)_{\hat\chi}$ and the action of 
$\rW_\chi$ on the fiber $i_\chi^*(\frak M(\mF\otimes\sign)|_{\calC(T)_{\hat\chi}})\is
i_\chi^*\frak M(\mF\otimes\sign)$ is trivial, by Lemma \ref{descent},
there exists $\calG\in D_{coh}^b(\calC(T)_{\hat\chi}//\rW_\chi)$ such that there 
is an isomorphism
\[\frak M(\mF\otimes\sign)|_{\calC(T)_{\hat\chi}}=q_{\hat\chi}^*\mG\in D^b_{coh}(\calC(T)_{\hat\chi}/\rW_\chi).\]
Since $\frak M$, $q_\chi^*$ and taking completions $(-)|_{\calC(T)_{\hat\chi}}$ are t-exact functors, we have 
\[\frak M(^p\tau_{\leq b}(\mF)\otimes\sign)|_{\calC(T)_{\hat\chi}}\is\tau_{\leq b}(\frak M(\mF\otimes\sign)|_{\calC(T)_{\hat\chi}})\is\tau_{\leq b}(q_\chi^*\mG)\is q_\chi^*(\tau_{\leq b}(\mG)).\]
Similarly, we have 
\[\frak M(^p\sH^b(\mF)\otimes\sign)|_{\calC(T)_{\hat\chi}}\is q_\chi^*\sH^b(\mG).\]
It follows that $\rW_\chi$ acts trivially on the 
fibers $i_\chi^*\frak M(^p\tau_{\leq b}(\mF)\otimes\sign)$ and 
$i_\chi^*\frak M(^p\sH^b(\mF)\otimes\sign)$ and, by Lemma \ref{chara of central},
we conclude that $^p\tau_{\leq b}(\mF)$ and $^p\sH^b(\mF)$ are $*$-central.
Now an induction argument on the (finite) number of non vanishing perverse cohomology sheaves of $\mF$
in \cite[Lemma 6.6]{C1} implies the Lemma.

The de Rham setting.
Let $\lambda\in\frak\ft(\bC)$ with maximal ideal 
$I_\lambda$ and let $\check\ft_\lambda$
be the completion at $\lambda$.
By Lemma \ref{finiteness}, the restriction 
$\frak M(\mF)|_{\check\ft_{\hat\lambda}}$ is a coherent complex on 
$\check\ft_{\hat\lambda}$.
Since 
$\mF$ is $*$-central and 
$\lambda$ is the unique closed point of $\check\ft_\lambda$, Lemma \ref{descent} implies 
\[\frak M(\mF\otimes\sign)|_{\check\ft_{\hat\lambda}}\is q_{\hat\lambda}^*\mH\in D_{coh}^b(\check\ft/\rW_{a,\lambda}),\]
where $q_{\hat\lambda}:\check\ft\to\check\ft//\rW_{a,\lambda}$
and $\mH$ is a coherent complex on $\check\ft//\rW_{a,\lambda}$. 
Now we can conclude by applying the 
same argument as in the $\ell$-adic setting where $\calC(T)_\chi$ and $\rW_\chi$
are replaced by
$\check\ft_\lambda$ and $\rW_{a,\lambda}$.

The case of strongly $*$-central complexes can be proved in the 
same way: applying the argument above where
$\rW_\chi$ and
$\rW_{a,\lambda}$
are replaced by $\rW_\chi'$ and $\rW_{a,\lambda}^{\on{ex}}$.
This completes the proof.

\end{proof}

\subsection{Convolution with $\cM_{\theta}$}\label{cons of gamma d mod}

\begin{proposition}\label{conv with Phi}
Let $\mF\in\sD_\rW(T)^\heartsuit$ be a strongly $*$-central perverse sheaf 
and let $\theta$ be a $\rW$-orbit through a tame character $\chi\in\calC(T)(F)$.
There is an isomorphism 
\[\Phi_{\mF}*\cM_{\theta}\is\oH^*(T,\mF\otimes\mL_{\chi^{-1}})\otimes\cM_{\theta}.\]

\end{proposition}
\begin{proof}
By Proposition \ref{conv with E_theta}, we have 
\beq\label{conv E_theta}
\mF*\mE_\theta\is\oH^*(T,\mF\otimes\mL_{\chi^{-1}})\otimes\mE_\theta\in\sD(T).
\eeq
Since $\on{Av}_*^U(\cM_{\theta})\is\mE_\theta$ is supported on 
$T=B/U\subset G/U$, 
Proposition \ref{prop of ind} implies
\beq\label{key iso}
\Ind_{T\subset B}^G(\mF)*\cM_{\theta}\is\Ind_{T\subset B}^G(\mF*\mE_{\theta})
\stackrel{}\is\oH^*(T,\mF\otimes\mL_{\chi^{-1}})\otimes\Ind_{T\subset B}^G(\mE_\theta).
\eeq
We claim that the isomorphism above is compatible with 
the natural $\rW$-actions.
Taking $\rW$-invariant on both sides of~\eqref{key iso}, 
we get
\[\Phi_{\mF}*\cM_{\theta}\is
\Ind_{T\subset B}^G(\mF)^\rW*\cM_{\theta}\stackrel{\eqref{key iso}}\is\oH^*(T,\mF\otimes\mL_{\chi^{-1}})\otimes\Ind_{T\subset B}^G(\mE_\theta)^\rW\is\oH^*(T,\mF\otimes\mL_{\chi^{-1}})\otimes\cM_{\theta}.\]
The proposition follows. 
\end{proof}
\subsubsection{
Proof of the claim}
Let us write $\Res=\Res_{T\subset B}^G$, $\Ind=\Ind_{T\subset B}^G$, 
$\cM=\cM_\theta$, $\mE=\mE_\theta$, 
and $V=\oH^*(T,\mF\otimes\mL_{\chi^{-1}})$.
Using Proposition \ref{prop of ind}
and the adjunction between $\Res$ and $\Ind$, 
it is straightforward to check that the following diagram commute
\beq\label{first diagram}
\xymatrix{F[\rW]\ar[r]^{(1)\ \ \ \ \ \ \ \ }\ar[dd]^{\Id}&\Hom(\Ind(\mF),\Ind(\mF))\ar[dd]^{(2)}\ar[r]^{*\cM\ \ \ \ \ \ }&\Hom(\Ind(\mF)*\mM,\Ind(\mF)*\mM)\ar[d]^{(4)}\\
&&\Hom(\Ind(\mF*\mE),\Ind(\mF)*\mM)\ar[d]^{(5)}\\
F[\rW]\ar[r]^{(3)\ \ \ \ \ \ \ \ \ \ }&\Hom(\mF,\Res\circ\Ind(\mF))\ar[r]^{*\mE\ \ \ \ }&\on{End}(\mF*\mE,\Res\circ\Ind(\mF)*\mE)}
\eeq
Here $(1)$ is given by the $\rW$-action on $\Ind(\mF)$,
$(2)$ is the adjunction map, $(3)$ is the composition $(2)\circ(1)$,
$(4)$ is induced by the isomorphism 
$\Ind(\mF)*\cM\is\Ind(\mF*\mE)$ in Proposition \ref{prop of ind}, and $(5)$
is induced by the canonical 
isomorphism 
\beq\label{mult of res}
\Res(\Ind(\mF)*\cM)\is\Res\circ\Ind(\mF)*\Res(\cM)\is\Res\circ\Ind(\mF)*\mE\footnote{The first isomorphism 
follows from the fact that, for any $\mF_1,\mF_2\in\sD(G/_\ad G)$ such that 
$\Av_*^U\mF_2$ is supported on $T$, we have canonical isomorphism 
$\Res(\mF_1*\mF_2)\is\Res(\mF_1)*\Res(\mF_2)$.}
\eeq
and the adjunction map.
On the other hand, the canonical map
\beq\label{res of key}
\Res\circ\Ind(\mF)*\mE\stackrel{\eqref{mult of res}}\is\Res(\Ind(\mF)*\cM)\stackrel{~\eqref{key iso}}\is\Res(V\otimes\Ind(\mE))\is V\otimes\Res\circ\Ind(\mE)
\eeq
is compatible with the natural $\rW$-actions. Indeed, Proposition \ref{Mackey formula}
and Proposition \ref{conv with E_theta} imply that there is a commutative diagram 
\[
\xymatrix{(F[\rW]\otimes\mF)*\mE\ar[r]\ar[d]^{\eqref{W-action 2}}&V\otimes(F[\rW]\otimes\mE)\ar[d]^{\eqref{W-action 2}}\\
\Res\circ\Ind(\mF)*\mE\ar[r]^{\eqref{res of key}}
&V\otimes(\Res\circ\Ind(\mE))}
\]
where the vertical arrows and the upper horizontal arrow
are compatible with the 
natural $\rW$-actions. It follows that the following diagram commute
\beq\label{second diagram}
\xymatrix{
F[\rW]\ar[dd]^{\Id}\ar[r]^{(3)\ \ \ \ \ \ \ }&\Hom(\mF,\Res\circ\Ind(\mF))\ar[r]^{*\mE\ \ \ \ \ }&\Hom(\mF*\mE,\Res\circ\Ind(\mF)*\mE))\ar[d]^{(6)}\\
&&\Hom(V\otimes\mE,V\otimes\Res\circ\Ind(\mE))\ar[d]^{(7)}\\
F[\rW]\ar[r]^{(8)\ \ \ \ \ \ \ }&\Hom(\Ind(\mE),\Ind(\mE))\ar[r]^{\Id_V\otimes\ \ \ \ \ }&\Hom(V\otimes\Ind(\mE),V\otimes\Ind(\mE))
}
\eeq
where $(6)$ is induced by~\eqref{res of key}, $(7)$ is the adjunction map, and $(8)$ is the $\rW$-action map.
Note that the composition of $(4),(5),(6),(7)$ gives rise to a map 
\beq\label{comp}
\End(\Ind(\mF)*\cM)\to\End(V\otimes\Ind(\mE))
\eeq 
which is 
equal to the map induced by the isomorphism~\eqref{key iso}, thus the commutativity of
\eqref{first diagram} and~\eqref{second diagram} implies that the following diagram commute 
\[\xymatrix{F[\rW]\ar[r]\ar[d]^{\Id}&\End(\Ind(\mF)*\cM)\ar[d]^{\eqref{comp}}\\
F[\rW]\ar[r]&\End(V\otimes\Ind(\mE))}\]
where the horizontal arrows are the $\rW$-action maps.
The claim follows.

\quash{
On the other hand, Proposition \ref{Mackey formula} and Proposition \ref{conv with E_theta} imply that the following diagram commute
\beq\label{second diagram}
\xymatrix{F[\rW]\ar[r]^{(3)\ \ \ \ \ \ \ \ }\ar[dddd]^{\Id}&\Hom(\mF,\Res\circ\Ind(\mF))\ar[r]^{*\cE\ \ \ \ \ \ }&\on{End}(\mF*\mE,\Res\circ\Ind(\mF)*\mE)\ar[d]^{(6)}\\
&&\Hom(\mF*\mE,(\bigoplus_{w\in\rW} w^*\mF)*\mE)\ar[d]^{(7)}\\
&&\Hom(V\otimes\mE,V\otimes(\bigoplus_{w\in\rW} w^*\mE))\ar[d]^{(8)}\\
&&\Hom(V\otimes\mE,V\otimes\Res\circ\Ind(\mE))\ar[d]^{(9)}\\
F[\rW]\ar[r]^{(10)\ \ \ \ \ \ \ \ }&\Hom(\Ind(\mE),\Ind(\mE))\ar[r]^{\Id_V\otimes(-)\ \ \ \ \ }&\Hom(V\otimes\Ind(\mE),V\otimes\Ind(\mE))}
\eeq
Here $(6)$ and $(8)$ are induced by the canonical isomorphisms 
in Proposition \ref{Mackey formula} (1),  
$(7)$ is induced by the isomorphisms in Proposition \ref{conv with E_theta},
$(9)$ is the adjunction map, and $(10)$ is induced by the $\rW$-action 
on $\Ind(\mE)$.
Note that the composition of maps from
$(4)$ to $(9)$ is equal 
}

\subsection{Proof of Theorem \ref{main}}
Since the Verdier duality interchanges strongly central complexes
with strongly $*$-central complexes, it suffices to verify Theorem \ref{main}
for strongly $*$-central complexes.
Let $\mF\in\sD_\rW(T)$ be a strongly $*$-central complex.
We need to show that the natural map 
\beq\label{restriction}
r:\on{Res}_{T\subset B}^G(\Phi_\mF)\ra\on{Av}_*^U(\Phi_\mF)
\eeq
is an isomorphism.
By Lemma \ref{reduction}, we can assume 
$\mF$ is a perverse sheaf.
We claim that, for any $\rW$-orbit $\theta=\rW\chi\subset\calC(T)(F)$, the convolution of $r$ with $\mE_\theta$ is an isomorphism
\beq\label{convolution}
\on{Res}_{T\subset B}^G(\Phi_\mF)*\mE_\theta\stackrel{\sim}\longrightarrow\on{Av}_*^U(\Phi_\mF)*\mE_\theta.
\eeq
For this, it is enough to show that 
$\on{Av}_*^U(\Phi_{\mF})*\mE_\theta$ is supported on $T$ and this follows from 
Theorem \ref{Key} and Proposition \ref{conv with Phi}. Indeed, we have  
\[\on{Av}^U_*(\Phi_\mF)*\mE_\theta\stackrel{\on{Thm\ \ref{Key}}}\is\on{Av}_*^U(\Phi_\mF)*\on{Av}^U_*(\cM_{\theta})\is
\on{Av}^U_*(\Phi_{\mF}*\cM_{\theta})\stackrel{\on{Prop} \ref{conv with Phi}}\is\oH^*(T,\mF\otimes\mL_{\chi^{-1}})\otimes\on{Av}_*^U(\cM_{\theta})\stackrel{\on{Thm\ \ref{Key}}}\is\]
\[\is\oH^*(T,\mF\otimes\mL_{\chi^{-1}})\otimes\mE_\theta.\]
Since
$\mE_\theta\is\bigoplus_{\chi\in\theta}\mE_\chi$, the isomorphism 
(\ref{convolution}) implies that 
the cone of the map in~\eqref{restriction}, denoted by $\on{cone}(r)$, satisfies 
$\on{cone}(r)*\mE_{\chi}=0$ for all $\chi\in\calC(T)(F)$. 
As $\mE_{\chi}$ is a local system on $T$ with generalized monodromy $\chi$,
that is, $\mE_{\chi}\otimes\mL_\xi$ is an unipotent local system, 
Lemma \ref{vanishing 2} and Lemma \ref{vanishing 1} below imply $\on{cone}(r)=0$. The theorem follows.

\subsection{Vanishing lemmas}\label{vanishing lemmas}
Let $X$ be a smooth variety with a free $T$ action $a:T\times X\ra X$.
For $\mL\in\sD(T)$ and $\mF\in\sD(X)$ we define 
$\mL*\mF:=a_*(\mL\boxtimes\mF)\in\sD(X)$.

\begin{lemma}\label{vanishing 2}
Let $\mL$ be a local system on $T$ with generalized monodromy $\chi\in\calC(T)(F)$, that is, 
$\mL\otimes\mL_\chi$ is an unipotent local system.  
Let $\mF\in\sD(X)$ and 
assume $\mL*\mF=0$. Then we have $\mL_\chi*\mF=0$.
\end{lemma}
\begin{proof}
There is a filtration $0=\mL^{(0)}\subset\mL^{(1)}\subset\cdot\cdot\cdot\subset\mL^{(k)}=\mL$
such that \[0\ra\mL^{(i-1)}\ra\mL^{(i)}\ra\mL^{(i)}/\mL^{(i-1)}\is\mL_\chi\ra 0.\]
Assume $\mL_\chi*\mF\neq 0$ and let $m$ be the smallest number such that $\sH^{\geq m}(\mL_\chi*\mF)=0$.
We claim that $\sH^{\geq m}(\mL^{(i)}*\mF)=0$
for $i=1,...,k$. The case $i=1$ is automatic since $\mL^{(1)}=\mL_\chi$.
For $i\leq k$, consider the 
distinguished triangle 
\[\mL^{(i-1)}*\mF\ra\mL^{(i)}*\mF\ra\mL_\chi*\mF\ra\mL^{(i-1)}*\mF[1]\]
induced from above short exact sequence.
Then for any $n\geq m$ we obtain an exact sequence
\[\sH^{n}(\mL^{(i-1)}*\mF)\to\sH^{n}(\mL^{(i)}*\mF)\to\sH^{n}(\mL_\chi*\mF)\to\]
By induction, the first and third terms are zero, and hence 
$\sH^{n}(\mL^{(i)}*\mF)=0$.
The claim follows.

Now since $\mL*\mF=0$, the 
distinguished triangle 
\[\mL^{(k-1)}*\mF\ra\mL*\mF\ra\mL_\chi*\mF\ra\mL^{(k-1)}*\mF[1]\]
implies 
\[\mL_\chi*\mF\is\mL^{(k-1)}*\mF[1].\]
Therefore we have $\sH^{m-1}(\mL_\chi*\mF)\is\sH^{m-1}(\mL^{(k-1)}*\mF[1])=\sH^{m}(\mL^{(k-1)}*\mF)=0$
which contradicts to the fact that $m$ is the smallest number such that $\sH^{\geq m}(\mL_\chi*\mF)=0$.
We are done.
\end{proof}

\begin{lemma}\label{vanishing 1}
Let $\mF\in\sD(X)$. If  
$\mL_\chi*\mF=0$ for all $\chi\in\calC(T)(F)$, then $\mF=0$.
\end{lemma}
\begin{proof}
Since $T$ acts freely on $X$ we have an
embedding $o_x:T\ra X, t\ra t\cdot x$. Moreover, by base change formulas, we have
\[\oH^*(T,\mL_{\chi^{-1}}\otimes^! o_x^!\mF)\is i_x^!(\mL_\chi*\mF)=0\]
for all $\chi\in\calC(T)(F)$. Here $i_x:x\ra X$ is the natural inclusion map.
By a result of Laumon
\cite[Proposition 3.4.5]{GL}, it implies $o_x^!\mF=0$ for all $x$.
The lemma follows.


\end{proof}

\end{document}